\documentclass[11pt,a4paper,reqno]{amsart}

\usepackage{color,calc,graphicx,amsfonts,amsthm,amscd,epsfig,psfrag,amsmath,amssymb,enumerate,dsfont}

\usepackage{amsmath,amssymb,graphics,epsfig,color,enumerate,psfrag}
\usepackage{verbatim}
\definecolor{refkey}{gray}{.75}
\definecolor{labelkey}{gray}{.5}

\newtheorem{Theorem}{Theorem}[section]
\newtheorem{TheoremA}{Theorem}
\newtheorem{Lemma}[Theorem]{Lemma}
\newtheorem{Proposition}[Theorem]{Proposition}

\newtheorem{Remark}[Theorem]{Remark}

\newtheorem{Claim}[Theorem]{Claim}

\newtheorem{Warning}[Theorem]{Warning}

\definecolor{darkgreen}{rgb}{0,0.4,0}

\definecolor{light}{gray}{.9}

\newcommand{\cC}{\ensuremath{\mathcal C}}

\newcommand{\cE}{\ensuremath{\mathcal E}}
\newcommand{\cF}{\ensuremath{\mathcal F}}

\newcommand{\E}{{\ensuremath{\mathbb E}} }
\newcommand{\bbE}{{\ensuremath{\mathbb E}} }

\newcommand{\bbI}{{\ensuremath{\mathbb I}} }

\newcommand{\bbL}{{\ensuremath{\mathbb L}} }

\newcommand{\N}{{\ensuremath{\mathbb N}} }

\renewcommand{\P}{{\ensuremath{\mathbb P}} }
\newcommand{\bbP}{{\ensuremath{\mathbb P}} }
\newcommand{\Q}{{\ensuremath{\mathbb Q}} }
\newcommand{\bbQ}{{\ensuremath{\mathbb Q}} }
\newcommand{\R}{{\ensuremath{\mathbb R}} }
\newcommand{\bbR}{{\ensuremath{\mathbb R}} }

\newcommand{\bbY}{{\ensuremath{\mathbb Y}} }
\newcommand{\Z}{{\ensuremath{\mathbb Z}} }
\newcommand{\bbZ}{{\ensuremath{\mathbb Z}} }

\usepackage{mathtools}

\let\a=\alpha \let\b=\beta   \let\d=\delta  \let\e=\varepsilon
 \let\g=\gamma       \let\l=\lambda
      \let\o=\omega      
  \let\s=\sigma \let\t=\tau   
  \let\z=\zeta
      
\let\O=\Omega 

\newcommand{\rosso}{\textcolor{black}}

\title[Einstein relation and linear response  in  1D Mott \rosso{variable--range} hopping]{Einstein relation and linear response  in  one--dimensional Mott variable--range hopping}

\author[A.\ Faggionato]{Alessandra Faggionato}
\address{Alessandra Faggionato.
  Dipartimento di Matematica, Universit\`a di Roma ``La Sapienza''.
 P.le Aldo Moro 2, 00185 Roma, Italy}
\email{faggiona@mat.uniroma1.it}

\author[N.\ Gantert]{Nina  Gantert}
\address{Nina Gantert. Fakult\"at f\"ur Mathematik,  Technische Universit\"at M\"unchen.  85748 Garching, Germany}
\email{gantert@ma.tum.de}

\author[M.\ Salvi]{Michele Salvi}
\address{Michele Salvi.  {Universit\'e Paris-Dauphine, PSL Research University, CNRS, [UMR 7534], CEREMADE, 75016 Paris, France}}
\email{salvi@ceremade.dauphine.fr}

\begin{document}

\begin{abstract}
We consider one-dimensional Mott variable-range hopping with a bias, and prove the linear response as well as the Einstein relation, under an assumption on the exponential moments of the distances between neighboring  points. In a previous paper \cite{FGS}  we gave conditions on ballisticity, and proved that in the ballistic case the environment viewed from the particle approaches,  for almost any  initial environment,  a given   steady state which is   absolutely continuous with respect to the original law of the environment.  Here, we show that  this bias--dependent steady state has  a derivative at zero  in terms of the bias (linear response), and use this result  to get the Einstein relation. 
 Our approach is new:  instead of using e.g. perturbation theory or  regeneration times, we show that the Radon-Nikodym derivative  of the  bias--dependent steady state with respect to the  equilibrium state in the unbiased case  satisfies an $L^p$-bound, $p>2$, uniformly for small  bias. This $L^p$-bound yields, by a general argument not involving our specific model, the statement about  the linear response.

\bigskip

\noindent  \emph{AMS  subject classification (2010 MSC)}: 
60K37, 
  60J25, 
  60G55, 
  82D30.

\smallskip
\noindent
\emph{Keywords}: Mott variable-range hopping, random conductance model, environment seen from the particle,
steady states,
linear response,
Einstein relation.

\thanks{The present work was financially supported  by  PRIN 20155PAWZB "Large Scale Random Structures" and the European Union's Horizon 2020 research and innovation programme under the Marie Sklodowska-Curie Grant agreement No 656047.
 }
\end{abstract}


\maketitle

\section{Introduction}

Mott variable--range hopping is a transport mechanism introduced by N.F. Mott \cite{M1,M2,M3,MD,POF,SE} to  model the 
phonon--assisted electron transport in disordered solids in the regime of strong Anderson localisation (e.g.~ doped
semiconductors and   doped  organic semiconductors). 

In the case of doped semiconductors,   atoms of some other material, called \emph{impurities}, are introduced into the solid at random locations $\{x_i\}$. 
 One can associate to each impurity  a random variable $E_i$ called \emph{energy mark}, the $E_i$'s  taking value in some finite interval $[-A,A]$. Due to the strong Anderson localisation, a single conduction electron is well described by a quantum wave--function localized around some impurity $x_i$ and $E_i$ is the associated energy in the ground state (to simplify the discussion we refer to spinless electrons). 
In Mott  variable--range  hopping an electron localized around $x_i$  jumps (by quantum tunneling) to another impurity site $x_k$, when $x_k$ is not occupied by any other electron,   with probability rate 
\begin{equation}\label{fermi}
C(\beta) \exp \Big \{-\frac{2}{\xi} |x_i-x_k| - \beta \{E_k -E_i \}_+\Big\} \,.
\end{equation}
Above,  $\b$ is  the inverse temperature, $\xi$ is the localization length, $\{v\}_+:=\max\{v,0\}$ and   the positive prefactor  $C(\beta)$ has a $\b$--dependence  which is  negligible w.r.t. the exponential decay in \eqref{fermi}.  Treating the localized electrons as classical particles,  the description is then given by an exclusion process on  the  sites $\{x_i\}$, with the above jump rates \eqref{fermi} when the exclusion constraint is satisfied. Calling $\eta$ a generic configuration in $\{0,1\}^{\{x_i\}}$, it then follows that   the disordered Bernoulli  distribution  $\mu$ on $\{0,1\}^{\{x_i\}}$ such that $ \mu (\eta_i)=\frac{{\rm e}^{-\b (E_i-\g)} }{ 1+{\rm e}^{-\b (E_i-\g)}  }$ is reversible for the exclusion process.  The chemical potential $\g$ is determined by the density of conduction electrons; equivalently - as usually done in the physical temperature  - we take $\gamma = 0$ at the cost of translating the energy (i.e. we take the Fermi energy level equal to zero).

We point out that the mathematical analysis of such an exclusion process is very demanding from a technical viewpoint due to site disorder. We refer to \cite{FMa,Q} for the derivation of the  hydrodynamic limit when the impurities are localized at the sites of $\bbZ^d$ and hopping is only between nearest--neighbor sites (from a physical viewpoint, the nearest--neighbor  assumption   leads to  a  good approximation of Mott variable--range hopping at not very small temperature).
Due to the these  technical difficulties, 
in the physical literature,  in the regime of   low density of conduction electrons  the above  exclusion process on $\{x_i\}$  is then approximated by independent  continuous time random walks (hence  one focuses on a single random walk), with  probability rate $r_{i,k}$  for a jump from $x_i$ to $x_k \not =x_i$ given by \eqref{fermi} times $\mu(\eta _{x_i}=1\,,\eta_{x_k}=0)$. Note that the last factor encodes the exclusion constraint.    The validity of this low density approximation has been indeed proved  
for the exclusion process with  nearest--neighbor jumps on $\bbZ^d$ (cf. \cite[Thm.~1]{Q}).

 It is  simple to check  (cf. \cite[Eq. (3.7)]{AHL}) that  in   the physically interesting low temperature regime (i.e. for large $\b$)  the resulting jump rate of the random walk behaves as 
\begin{equation}\label{fa_caldo}
r_{i,k}\approx C(\beta)  \exp\Big\{-\frac{2}{\xi} |x_i-x_k| - \frac{\b}{2} \bigl( |E_i |+|E_k | + |E_i-E_k|\bigr)\Big\} \,.
\end{equation}
In conclusion, considering the above approximations, Mott variable--range hopping consists of 
a random walk $(\bbY_t)$ in a random spatial and energetic environment given by $\{x_i\}$ and $\{E_i\}$  with jump rates   \eqref{fa_caldo}. We will consider here also a  generalization of the above jump rates (see eq. \eqref{def_rates} below).

The name variable--range hopping comes from the possibility of arbitrarily long jumps,  which are facilitated  (when 
 $\b$ is large)   if energetically convenient.  Indeed, it has been proved that long jumps contribute to most of the transport  in dimension $d\geq 2$  \cite{FM,FSS} but not in dimension $d=1$ \cite{CF}. The physical counterpart 
 of this feature is the anomalous     behavior of conductivity at low temperature for $d\geq 2$ \cite{POF,SE}, which has motivated the  introduction of   Mott variable--range hopping.  Indeed, for an isotropic medium, the conductivity $\s(\beta)$  is a multiple of the identity matrix and vanishes as $\b \to \infty$ as a stretched $\b$--exponential:
\begin{equation}\label{MES_big}
 \s(\b) \sim \exp  \bigl\{ - c \,\b^{\frac{\a+1}{\a+1+d}}\bigr\} \bbI  \end{equation}
if the energy marks are i.i.d. random variables with $P(|E_i| \in [E, E+dE])= c(\a) E^{\a}dE$ (these are the physically relevant energy distributions). On the other hand, in  dimension  $d=1$, the conductivity exhibits an Arrenhius--type decay (similarly to the nearest--neighbor case):
\begin{equation}\label{MES_1}
 \s(\b) \sim \exp  \bigl\{ - c\, \b \bigr\}\,.
 \end{equation}
 The decay \eqref{MES_big}   has been derived by heuristic arguments by Mott, Efros, Shklovskii  (see \cite{POF,SE} and references therein), afterwards refined by arguments involving   random  resistor networks and  percolation  \cite{AHL,MA}.  The decay   \eqref{MES_1}   has been derived by Kurkij\"arvi in terms of resistor networks \cite{Ku}.   A rigorous derivation of  upper and lower   bounds in agreement with  \eqref{MES_big} and \eqref{MES_1}  has been  achieved in \cite{FM,FSS} for $d\geq 2$ and in \cite{CF} for $d=1$. Strictly speaking, in \cite{CF,FM,FSS}  it has been shown that   the above random walk satisfies an invariance principle  and the asymptotic diffusion matrix   $D(\b)$  satisfies lower and upper bounds in agreement with the asymptotics in the r.h.s. of \eqref{MES_big} and \eqref{MES_1}.    Assuming the validity of the Einstein relation, i.e. $\s(\b)=\b D(\b)$, the same asymptotic is valid for the conductivity itself.
  We point out that, in  dimension $d=1$,  considering  shift--stationary and shift--ergodic point processes $\{x_i\}$ containing the origin,  
 the above result  on $D(\b)$  holds  if $\bbE\bigl[ {\rm e}^{Z_0}\bigr]<\infty$ where $Z_0= x_1-x_0 $, $x_1$ being the first point  right  to $x_0:= 0$  (cf. \cite[Thm.~1.1]{CF}). When $\bbE\bigl [{\rm e}^{Z_0}\bigr]=\infty$ the random walk is subdiffusive, i.e. $D(\b)=0$
 (cf. \cite[Thm.~1.2]{CF}).

  The present  work has  two main results:    Considering the above Mott variable--range hopping (also with more general jump rates) we develop  the   linear response theory  and derive  the  Einstein relation. As a byproduct, the latter, together   with \cite{CF}   completes the rigorous proof of  \eqref{MES_1}.
The presence of the external field of intensity $\l$  is modelled by inserting the term $\l \beta  (x_k-x_i)$ into the exponent in \eqref{fa_caldo}.  For simplicity of notation, and without loss of generality, we assume that  the localization length $\xi $ equals $2$. Then,
to have a well--defined random walk, one has to take $|\l|\b<1$.   As shown in \cite[Thm.~1, Thm.~2] {FGS},  if $\l\not =0 $ and  $\bbE[ {\rm e}^{(1-|\l|\b ) Z_0}]<\infty$, then  the random walk is ballistic (i.e. it has a strictly positive/negative  asymptotic velocity)
 and moreover the environment viewed from the walker admits an ergodic  invariant distribution $\bbQ_\l$ mutually absolutely continuous w.r.t.  the original 
law $\bbP$ of the environment.  Strictly speaking, the last statement is referred to the discrete--time version $(Y_n)_{n\geq 0}$  of the original continuous--time  Mott random walk $(\bbY_t)_{t\geq 0}$ (anyway, the latter can be obtained by a random time change from the former, which allows to extend asymptotic results from $Y_n$ to $\bbY_t$).  For $\l=0$ the result is still true with  $\bbQ_0$ having an explicit form and being reversible for the environment viewed from the walker.

The ergodicity of  $\bbQ_\l$  and its mutual  absolute continuity w.r.t.~$\bbP$, together with Birkhoff's  ergodic theorem, imply  in particular that, for any   bounded measurable function $f$, 
\begin{equation}\label{fase2}
\lim _{N \to \infty } \frac{1}{N} \sum_{n=0}^{N-1} f (  \o_n  ) = \bbQ_\l[ f] \qquad \text{a.s.}
\end{equation}
for $\bbP$--almost any environment $\o$, where $\o_n$ denotes the environment viewed from $Y_n$. Above, $\bbQ_\l[ f] $ denotes  the expectation of $f$ w.r.t. $\bbQ_\l$. In what follows, under the assumption that $\bbE[ {\rm e}^{p Z_0}] <\infty$, we show that the map $(-1,1) \ni \l \mapsto \bbQ_\l[f]\in \bbR$ is continuous if $p\geq 2$ (see Theorem \ref{teo_continuo}) and that  it is derivable at $\l=0$ if  $p>2$ and $f$ belongs to a precise $H_{-1}$ space  (see  Theorem \ref{teo_derivo}).  The derivative can moreover be expressed both in terms of the covariance of suitable additive functionals and in terms of potential forms (the first representation  is related to the Kipnis--Varadhan theory of additive functionals \cite{KV}, the second one  to homogenization theory \cite{Ko,MP1}).  We point out that similar issues concerning the behavior of the asymptotic steady state (characterized by \eqref{fase2}) for random walks in random environments have been addressed in  \cite{GGN} and \cite{MP2}. Finally, in Theorem \ref{teo_einstein} we  state the continuity in $\l$ of the asymptotic velocity of $(Y_n)$ and of $(\bbY_t)$ and the Einstein relation.

Two main technical difficulties lie behind linear response and Einstein relation:  Typically, in the biased case, the asymptotic steady state is not known explicitly and a limited information  on the speed of convergence to the steady state 
is available. 
A weaker form of the Einstein relation, which is often used as a
starting point, was proved in
\cite{LR}. Since then, the analysis of the  Einstein relation, the
steady states and the  linear response  for random walks in static/dynamic
random environments  have been addressed  in
\cite{ABF,BHOZ,GMP,GGN,G,KO1,KO2,LD,LOV,L1,L2,MZ,MP2} (the list is not
exhaustive).
The approach used here is different from the previous works: Although the  distribution $\bbQ_\l$ is not explicit, by refining the analysis of  \cite{FGS} we  prove that  the Radon--Nikodym derivative $\frac{d\bbQ_\l}{d\bbQ_0 }$ belongs to $L^p(\bbQ_0)$ if $\bbE\bigl[ {\rm e} ^{ pZ_0} \bigr] <\infty$ for some $p \geq 2$ (see Theorem \ref{serio_bis}). This result has been possible since $\bbQ_\l$ is indeed the weak limit as $\rho \to \infty$ of the asymptotic steady state of the   environment viewed from a $\rho$--cutoff version of $(Y_n)$, for which only  jumps between the first $\rho$ neighbors are admitted.  For the last $\rho$--parametrized asymptotic steady state  it is possible to express the  Radon--Nikodym derivative  w.r.t. $\bbP$ by a regeneration times  method developed already by Comets and Popov in \cite{CP} for random walks on $\bbZ$ with long jumps. This method is therefore very model--dependent. On the other hand, 
 having the above  bound on $\frac{d\bbQ_\l}{d\bbQ_0 }$,  one can derive Theorems \ref{teo_continuo}, \ref{teo_derivo} and \ref{teo_einstein} by a  general method that could be applied in other contexts as well.

\bigskip

\noindent{\bf Outline of the paper}: In Section~\ref{sec_mod_res} we describe the model, recall some previous results  and present our main theorems (Theorems~\ref{serio_bis}, \ref{teo_continuo}, \ref{teo_derivo} and \ref{teo_einstein}). Sections~\ref{sec_dim1} and \ref{biscotto1} are devoted to the proof of Theorem~\ref{serio_bis}. Theorem~\ref{teo_continuo} is proved in Section~\ref{dim_continuo}.
The proof of Theorem~\ref{teo_derivo} is split between  Sections~\ref{dim_teo_derivo} and \ref{cetriolo}. The proof of Theorem~\ref{teo_einstein} is  split between Sections~\ref{dim_ultimo} and \ref{dim_einstein}. Finally, in the Appendices~\ref{commentini}, \ref{poux} and \ref{appendino} we collect some technical results and proofs. 

\section{Models and main results}\label{sec_mod_res}

One-dimensional  Mott random walk  is a random walk in a random environment. 
 The environment $\o$ is given by 
   a  double--sided sequence  $(Z_k, E_k) _{k \in \bbZ}$,  with $Z_k \in (0,+\infty)$ and $E_k\in  \bbR$ for all $k \in \bbZ$.  
   We denote by {$\O= ( ( 0,+\infty )\times \bbR )^\bbZ$ the set  of all environments. Let $\bbP$ be a probability on $\O$, standing for the law of the environment, and let  $\bbE$   {be} the associated expectation.   Given $\ell \in \bbZ$, we define the  shifted environment $\t_\ell \o$ as $\t_\ell \o:=(Z_{k+\ell}, E_{k+\ell}) _{k \in \bbZ}$.  From now on, with \rosso{a} slight abuse of notation, we will denote by $Z_k, E_k$ also the random variables on $(\O, \bbP)$ such that $(Z_k(\o), E_k(\o))$ is  the $k$--th \rosso{coordinate} of the environment $\o$.

 Our main assumptions on the environment are the following:
\begin{itemize}
\item[(A1)] The {random} sequence $( Z_k,E_k)_{k \in \bbZ}$ is stationary and ergodic with respect to shifts;
\item[(A2)]   $\bbE[Z_0] $ is finite;
\item[(A3)]  $\bbP ( \o = \t_\ell \o)=0$  for  all $\ell \in \bbZ$;   
\item[(A4)]   There exists $  d>0$ such that 
 $\bbP(Z_0 \geq d)=1$.
\end{itemize}
\rosso{The} random environment can be thought of as a marked random point process \cite{DV,FKAS}. Indeed, we can associate to the 
 the   double--sided sequence  $(Z_k, E_k) _{k \in \bbZ}$ the point process $\{x_k\}_{k \in \bbZ}$ such that  $x_0=0$ and  $ x_{k+1}= x_k + Z_k$, marking each point $x_k$ with the value $E_k$. \rosso{We  introduce the map $\psi: \{x_k\}\to \bbZ$ defined as $\psi (x_k) =k$. }

 \smallskip

Given the environment $\o$  and $\l \in [0,1)$ we define the \emph{continuous--time Mott random walk} $(\bbY^\l_t)_{t\geq 0}$ as 
  the  random walk on $\{x_k\}_{k \in \bbZ}$  starting at $x_0=0$ with  probability rate for a jump from $x_i$ to $x_k\not= x_i$ given by
\begin{equation}\label{def_rates}
r^\l _{i,k}(\o):= \exp \{ -|x_i-x_k|+ \l (x_k-x_i)+ u(E_i, E_k)\}\,,
\end{equation}
with $u(\cdot, \cdot)$ a symmetric bounded  continuous  function.  It is convenient to set $r_{i,i} ^\l(\o):=0$. To have a well--defined random walk one needs to restrict to $|\l|<1$, and without loss of generality we assume $\l \in [0,1)$. 

We then define  \rosso{the}  \emph{discrete--time Mott random walk} $(Y^\l_n)_{n \geq 0}$ ($n$ varies in $\N:=\{0,1,\dots\}$)
as the jump process  associated to  $(\bbY^\l_t) $. In particular it is a random walk on $\{x_k\}_{k \in \bbZ}$  starting at $x_0=0$ with  probability for a jump from $x_i$ to $x_k $  given by
\begin{equation}\label{creperie}
p^\l _{i,k}(\o):= \frac{ r^\l _{i,k}(\o) }{\sum _{j\in\Z} r^\l _{i,j}(\o)}\,.
\end{equation}
Note that $p^\l_{0,0}\equiv 0$.
We denote by   $\varphi_\l$ the local drift of the random walk $(Y_n^\l)$, i.e.
\begin{equation}\label{loc_drift} \varphi_\l (\o):=  \sum _{k \in \bbZ} x_k \,p_{0,k}^\l (\o)\,.
\end{equation}

\begin{Warning}\label{paolino}
When $\l=0$ we usually omit the index $\l$ from the notation, writing simply  $\bbY_t$, $Y_n$,  $r _{i,k}(\o)$,  $p_{i,k}(\o)$, $\varphi(\o)$.
\end{Warning}

We now recall some results under the assumption  that $\l\in(0,1)$ and   $\bbE\bigl[{\rm e}^{ (1-\l) Z_0}  \bigr] <+\infty $ (cf. \cite[Thm.~1 and Thm.~2]{FGS}).
The asymptotic \rosso{velocities}
\begin{equation}\label{marta}
v_{Y}(\l):=\lim _{n \to \infty} \frac{Y^\l_n}{n} \qquad  v_{\bbY}(\l):=\lim _{t \to \infty} \frac{\bbY^\l_t}{t}
\end{equation} 
 exist  a.s. and  for $\bbP$--almost all realizations of the environment $\o$. The above asymptotic velocities   are deterministic and do  not depend on $\o$, they are  finite and strictly positive.
 The \emph{environment viewed from the discrete-time random  walk $(Y^\l_n)$}, i.e.~the process \rosso{$\bigl(\t_{ \psi (Y^\l_n) } \o\bigr)_{n \geq 0}$}, admits a unique   invariant and ergodic  distribution $\bbQ_\l$ which is absolutely continuous w.r.t.~ $\bbP$  (in \cite{FGS} uniqueness is not discussed: \rosso{Since} invariant ergodic distributions are mutually singular, $\bbQ_\l$ is the unique distribution fulfilling the above properties). Moreover, $\bbQ_\l$ and $\bbP$ are mutually  absolutely continuous.
 Finally (see also  \rosso{Appendix \ref{commentini}})  the asymptotic  velocities  $v_Y(\l)$ and $v_{\bbY}(\l)$  can be expressed as
   \begin{equation} 
  v_Y(\l)=  \bbQ_\l \bigl[  \varphi_\l \bigr] \;\; \text{ and }\;\; v_\bbY(\l)=\frac{  v_Y(\l)  }{  \bbQ_\l\Big[ 1/ ( \sum _{k\in\Z} r_{0,k} ^\l (\o) ) \Big]}\,.\label{adsl1}  
\end{equation}

\medskip

We recall some results concerning the unperturbed random walk $(Y_n)$ (\rosso{i.e.~with }$\l=0$). In this case the asymptotic velocities in \eqref{marta} still  exist  a.s.~and  for $\bbP$--almost all realizations of the environment $\o$, but they are zero: $v_Y(0)=v_\bbY (0)=0$ (cf.~\cite[Remark 2.1]{FGS}). Moreover,  the  environment viewed from the walker $(Y_n)$ has reversible measure $\bbQ_0$ defined as 
\begin{equation}\label{banana}
\bbQ_0 (d \o) = \frac{\pi (\o) }{ \bbE[ \pi ]} \bbP (d\o)\,,\qquad \pi(\o):=\sum_{k\in \bbZ} r_{0,k} (\o)\,.
\end{equation}
It is known  (cf.~\cite[Sec.~2]{CF})  that, when $\bbE\bigl[{\rm e}^{ Z_0}  \bigr] <\infty$, for $\bbP$--almost all  the realizations  of the environment $\o$ the    random walk $(Y_n)$ starting at the origin  converges, under diffusive rescaling, to a Brownian motion with positive diffusion  coefficient given by
\begin{equation}
D_Y= \inf _{g \in L^\infty (\bbQ_0)} \bbQ_0\Big[ \sum_{i\in \bbZ} p_{0,i}\bigl(x_i +\nabla_{i} g \bigr)^2 \Big]\,,
\end{equation}
where $\nabla _{i} g(\o):= g( \t_{i}\o)- g(\o) $ \rosso{(note that, since $\bbQ_0$ and $\bbP$ are mutually absolutely continuous, in formula (1.14) in \cite{CF} one can replace $L^\infty (\bbP)$   by  $L^\infty (\bbQ_0)$)}. Similarly  (cf. \cite[Thm.~1.1]{CF}) 
 $(\bbY_t)$ satisfies a quenched  functional CLT with diffusion coefficient
\begin{equation}
D_{\bbY} = \bbE[ \pi ]\,D_Y  \,.
\end{equation}

In order to present our results we need to  introduce the symmetric non--negative    operator $ -\bbL_0: L^2(\bbQ_0) \to L^2(\bbQ_0) $ with  $\bbL_0$   defined as 
\begin{equation}\label{def_bbL}
\bbL_0 f (\o) = \sum 
_{k \in \bbZ} p_{0,k}(\o) \bigl[ f(\t_k \o)- f(\o)\bigr]\,.
\end{equation}
 \rosso{We recall some basic facts on the spaces  $H_1$ and $H_{-1}$ associated  to the operator  $\bbL_0$   (cf.~\cite{demasi,KV,KLO})}.  In what follows we denote  the scalar product  in $L ^2(\bbQ_0)$    by $\langle \cdot, \cdot \rangle$. 
 \rosso{ The $H_1$ space}  is  given by the completion of  $L^2(\bbQ_0)$ endowed with the scalar product $\langle f, g \rangle_1:= \langle  f, -\bbL _0g \rangle$ and   $H_{-1}$ will denote the  space dual to $H_1$. In particular, $f \in L^2(\bbQ_0)$ belongs to $H_{-1}$ if and only if there exists a constant $C>0$ such that $|\langle f, g \rangle | 
 \leq C  \langle g ,-\bbL_0  g \rangle ^{1/2}$ for any $g\in L^2(\bbQ_0)$. Note that $\bbQ_0(f)=0$ for any $f \in  L^2(\bbQ_0) \cap H_{-1}$. Equivalently, denoting  \rosso{by} ${\rm e}_f(dx)$ the spectral measure associated to \rosso{ $f$ }and the operator $-\bbL_0$  (see e.g.~\cite{RS1}), 
 $f \in L^2(\bbQ_0)$ belongs to $H_{-1}$ if and only  $\int _{[0,\infty)}  \frac{1}{x} {\rm e}_f(dx)<\infty$.

\medskip 
We  can  now  present our main results.
Although  having a technical flavour, the following theorem  is indeed our starting point for the investigation of  the  continuity in $\l$ and the  linear response at $\l=0$ of the system, as explained in the introduction:

\begin{TheoremA}\label{serio_bis}  \rosso{Fix $\l_* \in (0,1)$ and suppose }that $\bbE\bigl[ {\rm e} ^{p Z_0 }\bigr] < +\infty $ for some   $p\geq 2$.  Then,  it holds 
 \begin{equation}
\sup _{\l \in [0,\l_*]} \Big \| \frac{d \bbQ_\l}{d \bbQ_0 }\Big \| _{L^p (\bbQ_0) }  <\infty\,.
 \end{equation}
\end{TheoremA}

\medskip
Our  next  result concerns the continuity in $\l$ of the expectation $\bbQ_\l(f)$.
\begin{TheoremA}\label{teo_continuo}
Suppose that $\bbE[{\rm e}^{p Z_0}] <\infty$  for some $p\geq 2$ and let $q$ be the \rosso{conjugate exponent of $p$}, i.e.~$q$ satisfies  $\frac{1}{p}+\frac{1}{q}=1$. 
Then,  for any $f \in L^q(\bbQ_0)$ and $\l \in [0,1)$,   \rosso{it holds that  $f \in L^1(\bbQ_\l)$ and   } the map 
 \begin{equation}\label{conti}
[0,1) \ni \l \mapsto  \bbQ_\l(f) \in \bbR 
\end{equation}
is continuous.
\end{TheoremA}
We point out that, for what concerns linear response \rosso{at $\l=0$}, only the continuity of the map \eqref{conti} at $\l=0$ plays some role. Anyway, our techniques allow to prove  continuity of   the map \eqref{conti} beyond the linear response regime.

\smallskip


Our next result concerns the derivative at $\l=0$ of   the map $\l \mapsto \bbQ_\l (f)$ 
for functions $f \in H_{-1}\cap L^2(\bbQ_0)$.  This  derivative can be represented both   as a suitable expectation involving a square integrable form and as  a covariance. To describe these representations we fix  some notation starting with the square integrable forms.

 We consider the space $\O\times \bbZ$ endowed with the  
measure $M$ defined by 
 \[ 
 M(u) = \bbQ_0 \Big[ \sum _{k\in\bbZ} p_{0,k} u(\cdot, k ) \Big] \,, \qquad \forall u :\O \times \bbZ\to \bbR \qquad \text{Borel, bounded}
 \,.\]
 A generic Borel function $u : \O \times \bbZ \to \bbR$ will be called a \emph{form}.   $L^2 ( \O \times \bbZ, M)$   is known as the space of \emph{square integrable forms}. Below, we will shorten the notation writing simply $L^2(M)$, and in general $L^p(M)$ for $ p$--integrable forms.  Given a  function $g =g(\o) $ we define 
 \begin{equation}\label{cirm}
  \nabla g (\o, k):= g(\t_k \o) - g(\o) \,.
 \end{equation}
If  $g \in L^2(\bbQ_0) $ then  $\nabla g\in L^2(M)$ (this follows \rosso{from  the identity} $\bbQ_0[ \sum _k p_{0,k} g(\t_k \cdot)^2]=
  \bbQ_0[g^2]$  due to the stationarity of $\bbQ_0$). The closure in $M$  of the subspace $\{ \nabla g \,:\, g\in L^2(\bbQ_0) \}$ forms the set of   the so called \emph{potential forms} (the orthogonal subspace is given by the so called \emph{solenoidal forms}).
  Take again $f \in H_{-1}\cap L^2(\bbQ_0)$ and, given $\e>0$, define $g^{f} _\e \in L^2(\bbQ_0)$ as the unique solution of the equation
 \begin{equation}\label{def_gigi}
 (\e -\bbL_0) g_{\e }^{f}= f \,.
 \end{equation}
 As discussed in Section \ref{dim_teo_derivo}, as $\e$ goes to zero the family of potential forms \rosso{$\nabla g_\e^{f}$} converges in $L^2(M)$ to a potential form $h^{f}$:
 \begin{equation}\label{hf} 
 h^{f} = \lim _{\e \downarrow 0 } \nabla g_\e ^{f} \qquad \text{ in } L^2(M)\,.
 \end{equation}

We now fix the notation that  will allow us to 
  state the second representation of  $\partial_{\l=0}\bbQ_\l (f)$  in terms of covariances. To this aim we write $(\o_n)$ for the environment viewed form the unperturbed walker $(Y_n)$, i.e.~\rosso{$\o_n := \t_{\psi(Y_n)} \o$} where $\o$ denotes the initial environment \rosso{(recall that $\psi(x_i)=i$)}. 
 Take now $f \in H_{-1}\cap L^2(\bbQ_0)$ . Due to \cite[Cor.~1.5]{KV} and Wold theorem, starting the process  $(\o_n)$ with distribution $\bbQ_0$, we have the following 
 weak convergence of 2d  random vectors 
\begin{equation}\label{jenner}
\frac{1}{\sqrt{n} } \Bigl( \sum_{j=0}^{n-1}  f(\o_j), \sum_{j=0}^{n-1} \varphi (\o_j) \Bigr) 		\stackrel{n\to \infty }{\rightarrow } (N^f , N^\varphi)
\end{equation}
for a suitable  2d gaussian vector $(N^f , N^\varphi)$   (with possibly degenerate diffusion matrix).   We recall that $\varphi$ denotes the local drift when $\l=0$ (cf. \eqref{loc_drift} and Warning \ref{paolino}).

We can now state our next main result:

\begin{TheoremA}\label{teo_derivo}
Suppose $\bbE[{\rm e}^{p Z_0}] <\infty$  for some $p>2$. Then, for any $f \in H_{-1}\cap L^2(\bbQ_0)$, the map $\l \mapsto  \bbQ_\l( f) $ is differentiable at  $\l=0$. Moreover  it holds
\begin{align}
\partial_{\l=0} \bbQ_\l (f)  
	& = \bbQ_0 \Bigl[ \sum_{k\in\bbZ} p_{0,k} (x_k -\varphi) h^{ f}(\cdot, k) \Bigr] \label{sole}\\
	& = -{\rm Cov}(N^f, N^\varphi) \,. \label{luna}
\end{align}
\end{TheoremA}

Starting from the above theorems one can derive  the continuity of the velocity and   the Einstein relation between velocity and diffusion coefficient both for $(Y_n)$ and for $(\bbY_t)$:
\begin{TheoremA}\label{teo_einstein} 
The following holds:
\begin{itemize}
\item[(i)] If  $\bbE[{\rm e}^{2 Z_0}] <\infty$, then  $v_Y(\l)$ and   $v_{\bbY}(\l)$ are continuous functions of $\l$;
\item[(ii)] If  $\bbE[{\rm e}^{p Z_0}] <\infty$  for some $p>2$, then the Einstein relation is fulfilled, i.e.
\begin{equation}\label{einstein}
\partial_{\l=0} v_Y (\l) = D_Y\qquad \text{ and } \qquad \partial _{\l=0}v_{\bbY} (\l) = D_\bbY\,.
\end{equation}
\end{itemize}
\end{TheoremA}
\begin{Remark}
We point out that in general the velocities $v_Y(\l)$ and $v_{\bbY}(\l)$ can have discontinuities. See \cite[Ex.~2 in  Sec.~2]{FGS} for an example.
\end{Remark}
If we make explicit the temperature dependence in the jump rates \eqref{def_rates} we would have 
\[ r^\l _{i,k}(\o):= \exp \{ -|x_i-x_k|+ \l \beta (x_k-x_i)+ \beta u(E_i, E_k)\}\,,\]
where $\l$ is the strength of the external field. 
Then equation \eqref{einstein} takes the more familiar (from a physical viewpoint) form
\[
\partial_{\l=0} v_Y (\l, \b)= \b   D_Y(\b) \text{ and } \partial _{\l=0}v_{\bbY} (\l,\b) = \b D_\bbY(\b)\,.
\]

\begin{Remark} In our treatment, and in particular 
in Theorems~\ref{teo_continuo},~\ref{teo_derivo} and~\ref{teo_einstein},  we have restricted our analysis to $\l \in [0,1)$. One can easily extend the above results to $\l \in (-1,1)$. Indeed, by taking a space reflection w.r.t. the origin, the resulting random environment still satisfies the main assumptions (A1),...,(A4) and the same exponential moment bounds as the original enviroment, while  random walks with negative bias become  random walks with positive bias. Hence, after taking a space reflection w.r.t. the origin,    one can apply the above theorems to study continuity for  $\l \in (-1,0]$ and derivability from the left at $\l=0$.  Noting that the left derivatives at $\l=0$ in Theorem \ref{teo_derivo} and \ref{teo_einstein} equal the right derivatives at $\l=0$, one recovers   that the claims in Theorems \ref{teo_continuo}, \ref{teo_derivo} and \ref{teo_einstein}  remain valid with $\l \in (-1,1)$.
\end{Remark}



\section{Proof of Theorem \ref{serio_bis}}\label{sec_dim1}
It is convenient to introduce the following notation for $i,j \in \bbZ$:
\begin{equation}\label{tropicale}
c^\l_{i,j}(\omega):=
\begin{cases}  {\rm e}^{-|x_j-x_i|+\l(x_i+x_j) + u (E_i,E_j) }&  \text{ if } i\not=j \,,\\
0  & \text{ otherwise}\,.
\end{cases}
\end{equation}
The above $c^\l_{i,j}(\omega) $ can be thought of as the  conductance associated to the \rosso{edge} $\{i,j\}$ and indeed the perturbed  \rosso{walk \rosso{$(Y^\l_n)$} is a random walk among the above   conductances}, since $p_{i,j}^\l (\o)= c_{i,j}^\l(\o) / \sum _{k \in \bbZ} c_{i,k}^\l(\o) $.

\smallskip

The proof of Theorem \ref{serio_bis} is an almost direct consequence of the following lemma:
\begin{Lemma} \label{meta} Fix $\l_* \in (0,1)$. \rosso{Then there exist positive constants $K,K_*$ such that, given  $\l \in (0,\l_*]$,} the  Radon--Nikodym derivative $\frac{d \bbQ_\l}{d\bbP\;} $ satisfies 
\begin{equation}\label{agognato}
\frac{d \bbQ_\l}{d\bbP}(\o) \leq  K \l  g(\o,\l) \,,
\end{equation} where
  \begin{equation}\label{def_g}  g(\o,\l):= K_0 (  c_{-1,0}^\l + c_{0,1}^\l)  \sum _{j=0}^\infty {\rm e}^{-2 \l x_j+ (1-\l)(x_{j+1}-x_j)}\,.
\end{equation}
\end{Lemma}

The proof of Lemma \ref{meta} requires a fine analysis of   Mott  random walk $(Y^\l_n)_{n\geq 0}$. 
We postpone it to the next section. Here  we show how to derive Theorem \ref{serio_bis} from Lemma \ref{meta}.

\begin{proof}[Proof of Theorem \ref{serio_bis}] It is enough to consider the case $\l \not =0$. 
The constants $c, C, C_*, C'$ appearing below are to be  \rosso{thought  independent } from $\l \in (0,\l_*]$ (they can depend on $\l_*$).
By \eqref{banana} and Lemma \ref{meta} we  can write 
\begin{equation}\label{natale}
\frac{d \bbQ_\l}{d \bbQ_0 } =\frac{d \bbQ_\l}{d \bbP } \frac{d \bbP}{d \bbQ_0 } 
	= \frac{\bbE[\pi]}{\pi}\frac{d \bbQ_\l}{d \bbP } 
	\leq\l C_*\frac{ c_{-1,0}^\l + c_{0,1}^\l}{\pi}\sum _{j=0}^\infty {\rm e}^{-2 \l x_j+ (1-\l)(x_{j+1}-x_j)}\,.
\end{equation}
Since \rosso{(recall the bounded function $u$ in \eqref{def_rates})}
\begin{align*}
 & \frac{ c_{-1,0}^\l + c_{0,1}^\l}{\pi}  {\rm e}^{ \rosso{-2\|u\|_\infty}}  \leq  \frac{\rosso{ {\rm e}^{ - |x_{-1}|+ \l x_{-1} }+ {\rm e} ^{-x_1+\l x_1}} }{\sum _{k\not=0} {\rm e}^{-|x_k|} }\leq 1 + {\rm e}^{\l x_1}\leq 2 {\rm e}^{\l x_1}\,,\\
& 
{\rm e}^{\l x_1}  \sum _{j=0}^\infty {\rm e}^{-2 \l x_j+ (1-\l)(x_{j+1}-x_j)}\leq {\rm e}^{  x_1} +
 \sum _{j=1}^\infty {\rm e}^{- \l x_j+ (1-\l)(x_{j+1}-x_j)} \leq {\rm e}^{ Z_0} +
 \sum _{j=1}^\infty  {\rm e}^{-\l d j + Z_j}\, ,
\end{align*}
from \eqref{natale} we get
$\frac{d \bbQ_\l}{d \bbQ_0 } \leq  2 {\rm e}^{\rosso{2 \| u\|_\infty }}    C_*  \l \sum _{j=0}^\infty  {\rm e}^{-\l d j + Z_j}$. As a consequence, \rosso{to conclude} 
 it is enough to prove that 
\begin{equation}\label{fine_bis0}
\bbQ_0\Big[ \bigl( \sum _{j=0}^\infty  {\rm e}^{-\l d j + Z_j}\bigr)^p \Big]\leq C/\l^p
\end{equation}
for some constant $C$.  To this aim let $q$ be the conjugate exponent such that $1/p+ 1/q=1$.
By the H\"older inequality we can bound
\[
 \sum _{j=0}^\infty  {\rm e}^{-\l d j + Z_j} 
 	\leq \Big(  \sum _{j=0}^\infty  {\rm e}^{-\frac{\l d  q }{2} j}  \Big) ^{\frac{1}{q} }     
 		\Big( \sum _{j=0}^\infty  {\rm e}^{-\frac{\l d p}{2} j +p  Z_j}\Big)^\frac{1}{p} =\Big( 1- {\rm e}^{-\frac{\l d  q }{2} }\Big)^{-\frac{1}{q}}     		\Big( \sum _{j=0}^\infty  {\rm e}^{-\frac{\l d p}{2} j +p  Z_j}\Big)^\frac{1}{p} \,.
\]
By using the above bound in \eqref{fine_bis0} we get 
\begin{equation*}
\begin{split}
\bbQ_0\Big[ \bigl( \sum _{j=1}^\infty  {\rm e}^{-\l d j + Z_j}\bigr)^p \Big] 
	& \leq \Big( 1- {\rm e}^{-\frac{\l d  q }{2} }\Big)^{-\frac{p}{q}}  
\Big( 1- {\rm e}^{-\frac{\l d  p }{2} }\Big)^{-1}  \bbQ_0 [ {\rm e}^{p Z_0}] 
\\
&\leq  \rosso{   (C'\l ) ^{-\frac{p}{q}}  (C' \l) ^{-1}} \bbQ_0 [ {\rm e}^{p Z_0}] = C \l ^{-\frac{p}{q}-1}= C \l ^{-p}\,,
\end{split}
\end{equation*}
thus implying \eqref{fine_bis0}.
\end{proof}

\section{Proof of Lemma \ref{meta}}\label{biscotto1}
In the first part of the  section we will improve a bound obtained in \cite{FGS}, see Proposition \ref{prop_pinolo}  below. This result will be essential to the proof of Lemma \ref{meta} (which will be carried out in Subsection \ref{biscotto2}).

\medskip

 In the rest of this section $\l\in (0,\l_*]$ is fixed once and for all and  is  omitted from the notation. In particular,  we write $(Y_n)$ for the biased 
discrete--time Mott random walk $(Y_n ^\l)$ and we write $c_{i,j}(\o)$ instead of $c^\l _{i,j}(\o)$ (cf.~ \eqref{tropicale}).
\rosso{As in \cite{FGS}, it will be convenient to consider the $\psi$--projection of $(Y_n)$ on the integers. We call $(X_n)$ the discrete-time random walk on $\bbZ$ such that $X_n=\psi (Y_n)$. } As already pointed out,  the  probability for a jump of $X_n$ from $i$ to $k$ is given by \eqref{creperie} which  equals $ \frac{c_{i,k} }{ \sum _{j \in \bbZ} c_{i,j}}$.

We further introduce a truncated version of $(X_n)$. We set  $\N_+:= \{1,2,3,\dots\}$.
For $\rho\in\N_+\cup\{+\infty\}$ we call $(X^\rho_n)$ the discrete-time random walk 
%
with jumping probabilities from $i$ to $j$ given by
\begin{equation}\label{salto}
\begin{cases} c _{i,j}(\o)/\sum_{k\in\bbZ}c _{i,k}(\o), & \mbox{if }0 < |i-j|\leq \rho\,, \\
0 & \mbox{if } |i-j|>\rho  \,,\\
1- \sum\limits_{j: |j-i| \leq \rho }
 c _{i,j}(\o)/\sum_{k\in\bbZ}c_{i,k}(\o)
 & \mbox{if } i=j\,.
\end{cases}
\end{equation}
Clearly the case $\rho=\infty$ corresponds to the random walk $(X_n)$. We write $P^{\o, \rho}_i$ for the law of $(X_n^\rho)$ starting at point $i\in\Z$ and $E^{\o, \rho}_i$ for the associated expectation.  In order to make the notation lighter, inside $P^{\o, \rho}_i(\cdot)$ and $E^{\o, \rho}_i[\cdot ]$ we will  sometimes write $X_n$ instead of $X_n^\rho$, when there will be no possibility of misunderstanding.

Call 
\begin{equation}\label{ananas}
T_i^\rho:=\inf\{n\geq 0:\,X_n^\rho\geq i\}
\end{equation} the first time the $\rho$-truncated random walk jumps over point $i\in\Z$ (also for $T^\rho_\cdot$ we will drop the $\rho$ super-index inside $P^{\o, \rho}_i(\cdot)$ and $E^{\o, \rho}_i[\cdot ]$).  A fundamental fact (cf.~ \cite[Lemma 3.16]{FGS})  is the following: One can find a positive $\e=\e(\l_*)$ independent from $\rho$, $\o$ and $\l \in (0,\l_*]$ such that 
 \begin{align}\label{piccolino}
P_k^{\omega,\rho}(X_{T_i}=i)\geq 2\varepsilon\qquad \forall k<i,\,\forall \rho\in\N_+\cup\{\infty\}\,.
\end{align}

\begin{Remark}\label{ricorda}  In \cite[Rem.~3.2]{FGS}  it is stated that all constants $K$'s and the constat $\e$ appearing in \cite[Sec.~3]{FGS} can be taken independent of $\l$ if $\l$  e.g. varies in $[0,1/2)$. As the reader can  easily check  the same still holds  as $\l $ varies in $[0,\l_*]$ for any fixed $\l_*$ in $(0,1)$ (\rosso{note that}  the above constants will depend on $\l_*$).
\end{Remark}

%
%

Given a subset $A \subset \bbZ$ we define $\t_A$ as the hitting time of the subset $A$, i.e.~$\t_A$ is the first nonnegative time for which the random walk is in $A$. 
For $A,B$ disjoint subsets of $\Z$, we define the effective $\rho$--conductance between $A$ and $B$ as
\begin{equation}\label{italia1}
C^{ \rho } _{\rm eff}(A, B)
	:=\min\Big\{\sum_{i<j: \,|i-j| \leq \rho }c_{i,j}(f(j)-f(i))^2:\,f|_A=0,\,f|_B=1\Big\}\,.
\end{equation}
The following technical fact provides a crucial estimate for the proof of Lemma \ref{meta}:
\begin{Lemma}\label{lemmadromedario}
For all $k\in\{1,...,\rho-1\}$,
\begin{align*}
P^{\omega,\rho}_k(\tau_0<\tau_{\,[\rho,\infty)})\geq 2 \varepsilon^3\frac{C_{\rm eff}^\rho(k,(-\infty,0])}{C_{\rm eff}^\rho(k , (-\infty,0]\cup [\rho, \infty))}\,.
\end{align*}
\end{Lemma}
\begin{proof}
For simplicity we will call $A:=(-\infty,0]$ and $B:=[\rho,\infty)$.
First of all notice that $P^{\omega,\rho}_k(\tau_0<\tau_{[\rho,\infty)})\geq 2\varepsilon P^{\omega,\rho}_k(\tau_{A}<\tau_{B})$. In fact,
\begin{align}\label{pocopoco}
P^{\omega,\rho}_k(\tau_0<\tau_B)
	&=\sum_{j\leq 0} P^{\omega,\rho}_k(\tau_0<\tau_B|\,\tau_A<\tau_B,\,X_{\tau_A}=j)
			P^{\omega,\rho}_k(\tau_A<\tau_B,\,X_{\tau_A}=j)\nonumber\\
	&=\sum_{j\leq 0} P^{\omega,\rho}_j(\tau_0<\tau_B)
		P^{\omega,\rho}_k(\tau_A<\tau_B,\,X_{\tau_A}=j)
		\geq 2\varepsilon P^{\omega,\rho}_k(\tau_A<\tau_B),
\end{align}
where in the last line we have used  that 
$ P^{\omega,\rho}_j(\tau_0<\tau_B)\geq P^{\omega,\rho}_j( X_{T_0}=0) 
\geq 2\varepsilon$, which follows from   \eqref{piccolino}. 
We can therefore focus on $P^{\omega,\rho}_k(\tau_{A}<\tau_{B})$. 

\smallskip

We consider now the following reduced Markov chain $(X_n')$ starting at $k$. Given $\omega\in\Omega$, $(X_n')$ is the random walk on the state space $\{0,...,\rho\}$ with conductances $c'_{i,j}= c'_{i,j}(\o)$ defined  by requiring that $c'_{i,j}= c'_{j,i}$ and that 
\begin{align*}
\rosso{c'_{i,j}}:=
\begin{cases}
c_{i,j}\qquad&\mbox{if }i, j\in\{1,...,\rho-1\}\,, \; i \not =j\,,\\
	\sum_{m: \, i-\rho\leq m \leq 0}c_{i,m}& \mbox{if } i\in\{1,...,\rho-1\}\,, \; j=0\,,\\
	\sum_{m:\, \rho\leq m\leq i+\rho}c_{i,m}& \mbox{if } i\in\{1,...,\rho-1\}\,, \; j=\rho \,,\\
0  & \mbox{if  } i=j\,. 
\end{cases}
\end{align*}
 We recall that, by definition,  the  probability for a transition from $i$ to $j$ in $\{0,1,\dots, \rho\}$ equals  $c_{i,j}'/ \pi'(i)$ where $\pi '(i)=\sum _{j: 0\leq j \leq\rho} c_{i,j}'$. Note that $\pi'$ is a reversible measure for $(X_n')$.


By a suitable coupling on an enlarged probability space (the probability of which will be denoted again by $P^{\omega,\rho}_k$) it holds 
\begin{align}\label{reduce}
P^{\omega,\rho}_k(\tau_A<\tau_B)=P^{\omega,\rho}_k(\tau'_0<\tau'_\rho),
\end{align}
where $\tau'_j$ is the first time  $(X_n')$ hits point $j$. In fact,  starting at $k$, if we ignore the times when $(X_n^\rho)$ does not move, 
$(X_n')$ and $(X_n^\rho)$ can be coupled 
in a way that guarantees that $X_n^\rho=X'_n$ until the moment when $X'_n$ touches $0$ or $\rho$. More precisely, one can couple the two random walks  to have that 
$( X_n'\,:\, 0 \leq n \leq \min \{\t'_0, \t'_\rho\} )$ equals the sequence of different visited sites  of the path  $( \phi( X_n)\,:\, 0 \leq n \leq \min \{\t_A, \t_B\} )$, where $\phi: \bbZ\to \{0,1,\dots, \rho\}$ is  defined as $\phi (i):=0$ for $i\leq 0$, $\phi(i)= \rho$ for $i\geq \rho$ and $\phi(i)=i$ otherwise.
The advantage of the above reduction is to have to deal now with a finite graph, so that we will be able to use classical results for resistor networks. 

\smallskip

As in \cite[proof of Fact 2]{BGP}, we call $t_0=0$ and $t_i$ the $i$-th time the walk $(X'_n)$ returns to the starting point $k$. We call the interval $[t_{i-1},t_i]$ the $i$-th excursion. For a set $D\subset\{0,...,\rho\}$ we call $V(i,D)$ the event that $(X_n')$ visits the set $D$ during the $i$-th excursion. We also call $\bar V(i,D)$ the event that set $D$ has been visited for the first time in the $i$-th excursion. 
Noticing now that the excursions are i.i.d., we can compute
\begin{align}\label{primavolta}
P^{\omega,\rho}_k(\tau'_0<\tau'_\rho)
	&=\sum_{i=1}^\infty P^{\omega,\rho}_k\big(\tau'_0<\tau'_\rho|\,\bar V(i,\{0,\rho\})\big)
		P^{\omega,\rho}_k(\bar V(i,\{0,\rho\}))\nonumber\\
	&=P^{\omega,\rho}_k\big(\tau'_0<\tau'_\rho|\,\bar V(1,\{0,\rho\})\big)
	=P^{\omega,\rho}_k\big(\tau'_0<\tau'_\rho|\, V(1,\{0,\rho\})\big),
\end{align}
so that we are left to estimate the probability that $0$ is visited before $\rho$ knowing that at least one of the two has been visisted during the first excursion. We see that
\begin{align}\label{briscola}
P^{\omega,\rho}_k\big(\tau'_0<\tau'_\rho|\, V(1,\{0,\rho\})\big)
	&\geq P^{\omega,\rho}_k\big(V(1,0)\cap V^c(1,\rho)|\, V(1,\{0,\rho\})\big)\nonumber\\
	&=\frac{P^{\omega,\rho}_k(V(1,0)\cap V^c(1,\rho))}{P^{\omega,\rho}_k(V(1,\{0,\rho\}))},
\end{align}
where $V^c(1,\rho)$ is the event that the walk does not visit $\rho$ during the first excursion.
We claim that 
\begin{align}\label{quattroundici}
P^{\omega,\rho}_k(V(1,0)\cap V^c(1,\rho))\geq \varepsilon^2\,P^{\omega,\rho}_k(V(1,0)).
\end{align}
To see why this is true, we first of all simplify the notation by setting $V(0):=V(1,0)$ and $V(\rho):=V(1,\rho)$ and write
\begin{align}\label{cornelius}
P^{\omega,\rho}_k(V(0)\cap V^c(\rho))= \alpha \,P^{\omega,\rho}_k(V(0))
\end{align}
with 
\begin{equation}
\begin{split}\label{tartaruga}
\alpha
	&:=\frac{P^{\omega,\rho}_k(V(0)\cap V^c(\rho))}{P^{\omega,\rho}_k(V(0))}
		=\frac{R_1}
		{R_1+R_2+R_3}
\end{split}
\end{equation}
with $R_1:=P^{\omega,\rho}_k(V(0)\cap V^c(\rho))$, $R_2:=P^{\omega,\rho}_k(V(0)\cap V(\rho),\, \tau_0'<\tau_\rho')$, $R_3:=P^{\omega,\rho}_k(V(0)\cap V(\rho),\, \tau_\rho'<\tau_0')$.

The proof of \eqref{quattroundici} is then based on the following bounds: 
\begin{align}
& R_3\leq\big(R_1+R_2\big)/(2\e)  \,, \label{chitarra1}\\
& R_2 \leq  R_1/(2\e)\,. \label{chitarra2}
\end{align}
Before proving \eqref{chitarra1} and \eqref{chitarra2} we explain how to derive \eqref{quattroundici} and conclude the proof of Lemma \ref{lemmadromedario}.

Trivially, \eqref{chitarra1} and \eqref{chitarra2}  imply that
 \[
\alpha\geq \frac{ R_1}{ (R_1+R_2) (1+\frac{1}{2\varepsilon})  }
	\geq \frac{1}{(1+\frac{1}{2\varepsilon})}\cdot
	\frac{R_1}
		{ R_1+\frac {1}{2\varepsilon} R_1}
	= \frac{1}{(1+\frac{1}{2\varepsilon})^2}\geq \varepsilon^2\,.
\]
In the last bound we have used that $\e \leq 1/2$ (cf.~ \eqref{piccolino}).
This together with \eqref{cornelius} gives \eqref{quattroundici}.

Putting now \eqref{quattroundici} into \eqref{briscola},  \eqref{briscola} into \eqref{primavolta}, we get
\begin{align}\label{donato}
P^{\omega,\rho}_k(\tau'_0<\tau'_\rho)\geq \varepsilon^2 \frac{P^{\omega,\rho}_k(V(1,0))}{P^{\omega,\rho}_k(V(1,\{0,\rho\}))}\,,
\end{align}
where we have restored the notation $V(i,D)$ for the event of having a visit to set $D$ during excursion $i$.
By a well-known formula (see, e.g., formula (5) in \cite{BGP}) we know that for each $D\subset\{0,...,\rho\}$
\begin{align}\label{sisapeva}
P^{\omega,\rho}_k(V(i,D))=\frac{C'_{\rm eff}(k,D)}{\pi'(k)},
\end{align}
where $C'_{\rm eff}(k,D)$ denotes the effective conductance between $k$ and $D$ in the reduced model.  More precisely, 
given disjoint subsets $E,F$ in $\{0,1,\dots, \rho\}$, we define
\begin{equation}\label{italia2}
C' _{\rm eff}(E, F)
	:=\min\Big\{\sum_{ i,j \,: \, 0\leq i<j \leq \rho  }c'_{i,j}(f(j)-f(i))^2:\,f|_E=0,\,f|_F=1\Big\}\,.
\end{equation}
As a byproduct of \eqref{donato} and \eqref{sisapeva} we get
\begin{equation}\label{solare}
P^{\omega,\rho}_k(\tau'_0<\tau'_\rho)
	\geq \varepsilon^2\frac{C'_{\rm eff}(k,0)}{C'_{\rm eff}(k,\{0, \rho\})} 
	\geq \varepsilon^2\frac{C^\rho_{\rm eff}(k,(-\infty,0])}{C^\rho_{\rm eff}(k,(-\infty,0]\cup [\rho,\infty))}.
\end{equation}
Let us explain the last bound. Given a function $f: \{0,1,\dots, \rho\}  \to \bbR$  and calling $\bar f$ its extension  on  $\bbZ$  such that 
$\bar f(z)=f(0)$ for all $z \leq 0$ and $\bar f(z)= f(\rho)$ for all \rosso{$z\geq \rho$}, it holds
\begin{equation} 
   \sum_{ i,j \,: \, 0\leq i<j \leq \rho  }c'_{i,j}(f(j)-f(i))^2= \sum_{i<j: \,|i-j| \leq \rho }c_{i,j}(\bar f(j)-\bar f(i))^2 \,.
\end{equation}
As a consequence,  by comparing the variational definitions of effective conductances given in 
\eqref{italia1} and \eqref{italia2}, one gets that  $C'_{\rm eff}(k,0)\geq C^\rho_{\rm eff}(k,(-\infty,0])$ and  
$C'_{\rm eff}(k,\{0, \rho\})=C^\rho_{\rm eff}(k,(-\infty,0]\cup [\rho,\infty))$, thus implying the last bound in \eqref{solare}.
Having \eqref{solare}, 
we finally use  \eqref{pocopoco} and \eqref{reduce} to get the lemma.

\medskip

We are left with the proof of \eqref{chitarra1} and \eqref{chitarra2}.

\medskip

$\bullet$ \emph{Proof of \eqref{chitarra1}}.
 We define $\tau_0(1):=\tau_0'$,  $\tau_\rho(1):=\tau_\rho'$ and 
 $\tau_{0,\rho}(1):=\min\{\tau_0(1),\,\tau_\rho(1)\}$. We also define  $x_*:= X'_{\tau_{0,\rho}(1)}$ \rosso{(note that $x_*$ equals  $0$ or $\rho$)}. Then, iteratively, for all $j\geq 1$ we define (see Figure \ref{fig1})
\begin{align*}
\tau_{0,\rho}(j+1)&:=\inf\{n:\,n>\tau_0(j),\,n>\tau_\rho(j),\,X'_n=x_*\}
\\
\tau_0(j+1)&:=
\begin{cases}
\tau_{0,\rho}(j+1) 					&\mbox{ if }x_*=0\\
\inf\{n:\, X_n=0,\,n>\tau_{0,\rho}(j+1)\} 	&\mbox{ if }x_*=\rho
\end{cases}
\\
\tau_\rho(j+1)&:=
\begin{cases}
\inf\{n:\, X_n=\rho,\,n>\tau_{0,\rho}(j+1)\} 	&\mbox{ if }x_*=0\\
\tau_{0,\rho}(j+1) 					&\mbox{ if }x_*=\rho\,.
\end{cases}
\end{align*}

Notice that, almost surely, either $\tau_0(1)<\tau_\rho(1)<\tau_0(2)<\tau_\rho(2)<...$ or $\tau_\rho(1)<\tau_0(1)<\tau_\rho(2)<\tau_0(2)<...$.
We also define $\tau_k^+$ as the first time the random walk started in $k$ returns to $k$. 
Notice that all the $\tau_\cdot(\cdot)$'s and $\tau_k^+$ are stopping times. We decompose
\begin{align}\label{decomposizione}
R_3:=P^{\omega,\rho}_k(V(0)\cap V(\rho),\, \rosso{\tau_\rho'<\tau_0'})
	=P^{\omega,\rho}_k(\tau_\rho(1)<\tau_0(1)<\tau_k^+)
	=\sum_{i=1}^\infty A_i+\sum_{i=2}^\infty B_i\,,
\end{align}
where
\begin{align*}
A_i&:=P^{\omega,\rho}_k(\tau_\rho(1)<\tau_0(1)<\tau_\rho(2)<\dots <\tau_0(i)<\tau_k^+<\tau_\rho(i+1))
\nonumber\\
B_i&:=P^{\omega,\rho}_k(\tau_\rho(1)<\tau_0(1)<\tau_\rho(2)<\dots <\tau_\rho(i)<\tau_k^+<\tau_0(i))\,.
\end{align*}
We first focus on the terms of the form $A_i$.

\begin{Claim}\label{vongola}
It holds
\begin{align}
A_i&\leq \frac 1{2\varepsilon}C_i, \qquad\;\;\;\; i\geq 1 \label{bicchiere1}\\
B_i&\leq \frac{1}{2\varepsilon}D_{i-1},  \qquad i\geq 2\label{bicchiere2}
\end{align}
where
 \begin{align*}
C_{i}&:=P^{\omega,\rho}_k(\tau_0(1)<\tau_\rho(1)<\dots <\tau_0(i)<\tau_k^+<\tau_\rho(i)),\\
D_{i}&:=P^{\omega,\rho}_k(\tau_0(1)<\tau_\rho(1)<\dots <\tau_0(i)<\tau_\rho(i)<\tau_k^+<\tau_0(i+1)).
\end{align*}

\end{Claim} 
\begin{proof}[Proof of the Claim]


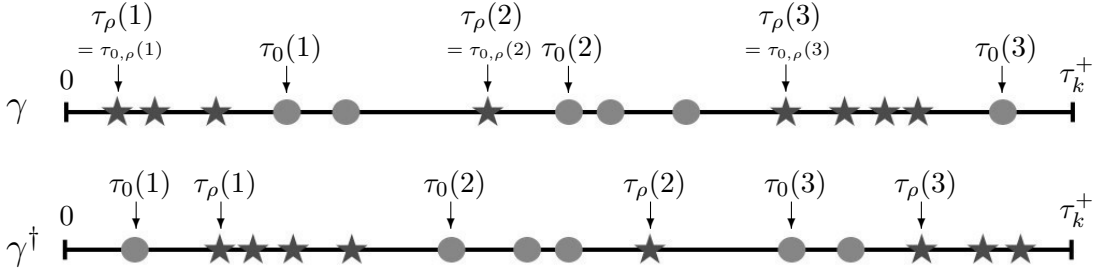
\begin{figure}
  \centering
  \setlength{\unitlength}{0.12\textwidth}  
  \begin{picture}(7.8,2.3)
    \put(0,0){\includegraphics[scale=0.4]{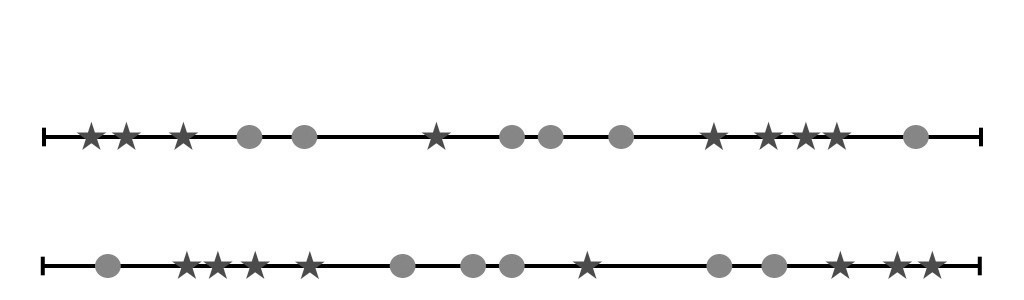}}
	 \put(-0.1,0.2){\Large$\gamma^\dag$}
	 \put(0.29,0.45){$0$}	 
	 \put(7.6,0.5){$\tau_k^+$}
	 \put(0.85,0.62){\vector(0,-1){0.2}\put(-0.2,0.07){$\tau_0(1)$}}	 
	 \put(1.47,0.62){\vector(0,-1){0.2}\put(-0.2,0.07){$\tau_\rho(1)$}}
 	 \put(3.15,0.62){\vector(0,-1){0.2}\put(-0.2,0.07){$\tau_0(2)$}}
 	 \put(4.61,0.62){\vector(0,-1){0.2}\put(-0.2,0.07){$\tau_\rho(2)$}}
 	 \put(5.64,0.62){\vector(0,-1){0.2}\put(-0.2,0.07){$\tau_0(3)$}}
	 \put(6.59,0.62){\vector(0,-1){0.2}\put(-0.2,0.07){$\tau_\rho(3)$}}
	 \put(-0.1,1.25){\Large $\gamma$}
 	 \put(0.29,1.45){$0$}
	 \put(0.72,1.63){\vector(0,-1){0.2}\put(-0.3,0.07){\tiny$=\tau_{0,\rho}(1)$}}
	 \put(0.72,1.86){\put(-0.2,0.07){$ \tau_\rho(1)$}}
	 \put(1.95,1.62){\vector(0,-1){0.2}\put(-0.2,0.07){$\tau_0(1)$}}
	 \put(3.42,1.63){\vector(0,-1){0.2}\put(-0.3,0.07){\tiny$=\tau_{0,\rho}(2)$}}
	 \put(3.42,1.86){\put(-0.2,0.07){$ \tau_\rho(2)$}}
 	 \put(4.01,1.62){\vector(0,-1){0.2}\put(-0.2,0.07){$\tau_0(2)$}}
	 \put(5.6,1.63){\vector(0,-1){0.2}\put(-0.3,0.07){\tiny$=\tau_{0,\rho}(3)$}}
	 \put(5.6,1.86){\put(-0.2,0.07){$ \tau_\rho(3)$}}
	 \put(7.18,1.62){\vector(0,-1){0.2}\put(-0.2,0.07){$\tau_0(3)$}}
	 \put(7.6,1.5){$\tau_k^+$}
  \end{picture}
 \caption{\textit{$\gamma$ \rosso{corresponds to }the trajectory of an excursion from $k$ to $k$ associated to the probability $A_i$, $i=3$. $\gamma^\dag$ is the time-reversed trajectory. Balls denote times when the random walk hits $0$, while stars denote times when it hits $\rho$. 
}}\label{fig1}
\end{figure}

We start with \eqref{bicchiere1}.
By reversibility (just decompose the event on all the possible trajectories of the random walk and then use the detailed balance equations, see Figure \ref{fig1}),
\begin{align}\label{boundA}
A_i = P^{\omega,\rho}_k(\tau_0(1)<\tau_\rho(1)<\dots <\tau_\rho(i)<\tau_k^+<\tau_0(i+1)).
\end{align}
On the other hand, we have
\begin{align}\label{boundC}
C_{i}
	&\geq P^{\omega,\rho}_k(\tau_0(1)<\tau_\rho(1)<\dots <\tau_0(i)<\tau_k^+ ,\, \xi(0,i) )\nonumber\\
	&=P^{\omega,\rho}_k(\tau_0(1)<\tau_\rho(1)<\dots <\tau_0(i)<\tau_k^+)\, P^{\omega,\rho}_0(X'_{T_k^\rho}=k )\nonumber\\
	&\geq 2\varepsilon \,P^{\omega,\rho}_k(\tau_0(1)<\tau_\rho(1)<\dots <\tau_0(i)<\tau_k^+)\,\geq 2\varepsilon\, A_i,
\end{align}
where the event $\xi(0,i)$ is defined as $\xi(0,i):=\{$the first time after $\tau_0(i)$ that the random walk tries to overjump the point $k$, it actually lands on $k$$\}$. The first inequality is trivial. For the second line, we can apply the strong Markov property at the stopping time $\tau_0(i)$ observing that  the event $\{\tau_0(1)<\tau_\rho(1)<\dots <\tau_0(i)<\tau_k^+\}$ is in the $\sigma$-algebra generated by the process up to time $\tau_0(i)$. Finally, for the last line we first notice that \eqref{piccolino} is also valid for the random walk $(X'_n)$ and then use \eqref{boundA}. This gives \eqref{bicchiere1}.

We move to the proof of \eqref{bicchiere2}.
Clearly
\begin{align}\label{boundB}
B_i \leq P^{\omega,\rho}_k(\tau_\rho(1)<\tau_0(1)<\dots <\tau_0(i-1)<\tau_k^+)\,.
\end{align}
On the other hand
\begin{align}\label{boundD}
D_{i-1}
	&=P^{\omega,\rho}_k(\tau_\rho(1)<\tau_0(1)<\dots <\tau_0(i-1)<\tau_k^+<\tau_\rho(i))\nonumber\\
	&\geq P^{\omega,\rho}_k(\tau_\rho(1)<\tau_0(1)<\dots <\tau_0(i-1)<\tau_k^+,\, \xi(0,i-1))\nonumber\\
	&\geq 2\varepsilon \,P^{\omega,\rho}_k(\tau_\rho(1)<\tau_0(1)<\dots <\tau_0(i-1)<\tau_k^+)\,,
\end{align}
where for the first line we have used again the reversibility of the process, in the second and third line we have used the same arguments as for the proof of \eqref{bicchiere1}.
\eqref{boundB} and \eqref{boundD} together show that $B_i\leq \frac 1{2\varepsilon} D_{i-1}$.
\end{proof}

We come back to \eqref{decomposizione}. Thanks to the above claim, we have
\begin{align*}
R_3
	&\leq \frac 1{2\varepsilon}\Big(C_1+\sum_{i=2}^\infty C_i+ \sum_{i=1}^\infty D_i\Big)\\
	&=\frac 1{2\varepsilon}\Big(P^{\omega,\rho}_k(V(0)\cap V^c(\rho))+P^{\omega,\rho}_k(V(0)\cap V(\rho),\, \tau_0'<\tau_\rho')\Big)=\frac 1{2\varepsilon}\big(R_1+R_2\big)
\end{align*}
as we wished, since $C_1=P^{\omega,\rho}_k(V(0)\cap V^c(\rho))$ and $\sum_{i=2}^\infty C_i+ \sum_{i=1}^\infty D_i$ is a decomposition of the probability of the event $\{V(0)\cap V(\rho),\, \tau_0'<\tau_\rho'\}$ in a similar fashion as in \eqref{decomposizione}.


\medskip

$\bullet$ \emph{Proof of \eqref{chitarra2}}.
We notice that
\begin{align*}
R_1=P^{\omega,\rho}_k(V(0)\cap V^c(\rho))
	&=P^{\omega,\rho}_k(\tau_0(1)<\tau_k^+<\tau_\rho(1))\\
	&\geq P^{\omega,\rho}_k(\tau_0(1)<\tau_k^+,\,\tau_0(1)<\tau_\rho(1),\,\xi(0,1))\\
	&\geq 2\varepsilon P^{\omega,\rho}_k(\tau_0(1)<\tau_k^+,\,\tau_0(1)<\tau_\rho(1))\\
	&\geq 2\varepsilon P^{\omega,\rho}_k(V(0)\cap V(\rho),\, \tau_0'<\tau_\rho')=2\varepsilon\,R_2\,,
\end{align*}
where we have used the event $\xi(0,1)$ introduced in the proof of Claim \ref{vongola} and the same argument based on \eqref{piccolino} therein.
%
\end{proof}

Having Lemma \ref{lemmadromedario} we can prove the following lower \rosso{bound} on the expected value of $T_\rho$, which refines  that of Lemma 4.3 in \cite{FGS}:

\begin{Proposition}\label{prop_pinolo} Fix $\l_*\in (0,1)$. Then  there exist 
 constants $C_1,C_2>0$,  independent of  \rosso{$\lambda\in(0,\l_*]$}  and  of $\rho \in {\N}_+\cup \{+\infty\}$, such that 
\begin{align*}
\bbE E^{\omega,\rho}_0[T_\rho]\geq  C_1 \frac{\rho}{\lambda} -C_2 \frac{1}{\l^2}\,.
\end{align*}
\end{Proposition}
\begin{proof}
Formula (3.22) in \cite{B06} reads in our case as
\begin{align*}
E^{\omega,\rho}_0[T_\rho]
	&=\frac{1}{C^\rho_{\rm eff}(0,[\rho,\infty))}\sum_{k<\rho} \Big(\sum_{j\in\Z}c_{k,j}\Big) P^{\omega,\rho}_k(\tau_0<\tau_{[\rho,\infty)})\,,
\end{align*}
where $k \mapsto \sum_{j\in\Z}c_{k,j}$ is a reversible measure for the $\rho$-truncated random walk for each $\rho$.
Hence,
\begin{align}\label{jupiter}
E^{\omega,\rho}_0[T_\rho]
	&\geq \frac{1}{C^\rho_{\rm eff}(0,[\rho,\infty))}\sum_{0<k<\rho} \Big(\sum_{j\in\Z}c_{k,j}\Big)  P^{\omega,\rho}_k(\tau_0<\tau_{[\rho,\infty)})\nonumber\\ 
	&\geq C_3 \sum_{0<k<\rho}\rosso{c_{k,k+1}}\frac{C_{\rm eff}^1(k,(-\infty,0])}{C^1_{\rm eff}(0,[\rho,\infty))C_{\rm eff}^1(k,(-\infty,0]\cup[\rho, \infty))}\,,
\end{align}
where $C_3$ is a strictly positive constant independent of $\rho$, $\omega$ and $\lambda$ as $\l$ varies  in $(0,\l_*]$ (as all the constants of the form $C_i$ that will appear in what follows). For the last line in \eqref{jupiter} we have used Lemma \ref{lemmadromedario} and the bounds 
\begin{align*}
C_{\rm eff}^1(A,B)\leq C_{\rm eff}^\rho(A,B)\leq c\cdot C_{\rm eff}^1(A,B),
\end{align*}
for some universal constant $c\geq1$. The above bounds follow from  \cite[Prop.~3.4]{FGS}. The fact that $c$ can be taken uniformly in $\l\in [0,\l_*]$ follows from  \cite[Rem.~3.2]{FGS}  and Remark \ref{ricorda}.

Writing for simplicity $c_j:=c_{j,j+1}$, we explicitly calculate
\begin{align}\label{cocco}
&\frac{C_{\rm eff}^1(k,(-\infty,0])}{C^1_{\rm eff}(0,[\rho,\infty))C_{\rm eff}^1(k,(-\infty,0]\cup[\rho, \infty))}\nonumber\\
&\qquad\qquad\qquad\qquad
	=\frac{\big(\sum_{j=0}^{k-1}\frac{1}{c_j}\big)^{-1}}{\big(\sum_{j=0}^{\rho-1}\frac{1}{c_j}\big)^{-1}\Big(\big(\sum_{j=k}^{\rho-1}\frac{1}{c_{j}}\Big)^{-1}+\big(\sum_{j=0}^{k-1}\frac{1}{c_j}\big)^{-1}\Big)}
	=\sum_{j=k}^{\rho-1}\frac{1}{c_j}\,.
\end{align}
Therefore, by taking the expectation w.r.t.~the environment in \eqref{jupiter}, we obtain
\begin{align}\label{peperoni}
\bbE E^{\omega,\rho}_0[T_\rho]
	&\geq C_3\,\bbE\Big[\sum_{0<k<\rho}\rosso{c_k}\sum_{j=k}^{\rho-1}\frac{1}{c_j}\Big]\nonumber\\
	&\geq C_3  \rosso{ {\rm e}^{ -2 \|u\|_\infty}}  \,\bbE\Big[ \sum_{0<k<\rho}{\rm e}^{-(1-\l)Z_k+2\l(Z_0+...+Z_{k-1})}\sum_{j=k}^{\rho-1}{\rm e}^{(1-\l)Z_j-2\l(Z_0+...+Z_{j-1})}\Big]\nonumber\\
	&\geq C_4\,\Big(\rho+ \sum_{0<k<\rho}\sum_{j=k+1}^{\rho-1}\bbE[ {\rm e}^{-(1-\l)Z_k-2\l(Z_k+...+Z_{j-1})}]\Big)\nonumber\\
	&\geq C_5\, \sum_{0<k<\rho}\sum_{j=k+1}^{\rho-1} {\rm e}^{-2\l\bbE[Z_0](j-k)}\,,
\end{align}
where in the third line $\rho$ comes from the case $j=k$ and in the last line we have used Jensen's inequality and the fact that ${\rm e}^{-(1-\l)\bbE[Z_0]}$ is bigger than a  constant independent from $\l$.
We call now $A:={\rm e}^{-2\l\bbE[Z_0]}<1$ and calculate
\begin{align*}
\sum_{0<k<\rho}\sum_{j=k+1}^{\rho-1}A^{j-k} & = \sum_{0<k<\rho}\sum_{m=1}^{\rho-k-1}A^{m}
	=\sum_{0<k<\rho}\frac{A-A^{\rho-k}}{1-A}=\sum_{0<k<\rho}\frac{A-A^{k}}{1-A}\\
	& 
	=(\rho-1)\frac{ A }{1-A}-\frac{ A - A^{\rho} }{(1-A)^2}
	\geq C_6 \Big(\frac{\rho}{1-A}-\frac{1}{(1-A)^2}\Big)\,.
\end{align*}
We can then continue the chain of inequalities of \eqref{peperoni}:
\begin{align*}
\bbE E^{\omega,\rho}_0[T_\rho]
	&\geq C_7\,\Big(\frac{\rho}{1-A}-\frac{1}{(1-A)^2}\Big)
	\geq C_1\,\frac{\rho}{\lambda}-C_2\,\frac{1}{\lambda^2}\,,
\end{align*}
which is the statement of the proposition. Here we have used the fact that 
\[ 0 < \inf _{\l \in (0,\l_*]} \frac{\l}{ 1- {\rm e}^{-2\l\bbE[Z_0]}}<  \sup _{\l \in (0,\l_*]} \frac{\l}{ 1- {\rm e}^{-2\l\bbE[Z_0]}}<+\infty\,,\]
which follows from the fact that the the function $\frac{\l}{1-A}=\frac{\l}{ 1- {\rm e}^{-2\l\bbE[Z_0]}}$ can be extended to a continuous \rosso{strictly positive function} on the compact interval $[0,\l_*]$. 
\end{proof}

\subsection{Proof of Lemma \ref{meta}}\label{biscotto2}
 With Proposition \ref{prop_pinolo} we can finally prove Lemma \ref{meta}. We first stress that below all constants of type $C,K$ can depend on $\l_*$, but do not depend on the chosen parameter $\l \in (0,\l_*]$.
We recall that, in \cite{FGS}, for a given $\rho\in\N\cup\{+\infty\}$, one calls $\bbQ^{ \rho}$ the asymptotic invariant distribution for the environment viewed from  the  $\rho$--truncated random walk $(X_n ^\rho)$, when an external drift of intensity $\l$ (here implicit in the notation) is applied (the case $\rho=\infty$  corresponds again to the random walk $(X_n)$ without cut-off, and $\bbQ^\infty=\bbQ_\l$)\rosso{.}
In \cite{FGS} it is shown that $\bbQ^\rho$ is absolutely continuous to $\bbP$. In order to describe the Radon--Nikodym derivative $\frac{ d\bbQ^\rho}{d \bbP}$ we have to introduce an auxiliary process.
We let $\zeta=(\zeta_1,\zeta_2,...)$ be a sequence of i.i.d.~Bernoulli random variables of parameter $\varepsilon$, where $\varepsilon$ is the same appearing in \eqref{piccolino}.
We call $P$  the law of $\zeta$ and $E$ the relative expectation. As detailed in \cite[Sec.~4]{FGS} adapting a construction in \cite{CP},  one can couple $\zeta$ and the random walk $(X_n^\rho)$ so that if $\zeta_j=1$ for some $j\in\N$, then $X^\rho_{T_{j\rho}^\rho}=j\rho$ (see \eqref{ananas}). In \cite[Eq.~(46) and Eq.~(47)]{FGS} one has the precise construction of the quenched probability $P^{\omega,\rho,\zeta}_0$ for the random walk once the sequence $\zeta$ has been fixed.  $E^{\omega,\rho,\zeta}_0$ is the associated expectation.
The Radon--Nikodym derivative for the environment viewed from the $\rho$-truncated walk w.r.t.~the original measure of the environment $\bbP$ is given by (cf.~\cite[Eq.~(63)]{FGS})
\begin{equation}\label{umbria}
\frac{ d\bbQ^\rho}{d \bbP} (\o) =\frac{1}{\bbE E[ E_0^{\o,\zeta, \rho }[T_{\ell_1\rho}]]}\sum _{k \in \bbZ} E  E^{\t_{-k} \o, \z, \rho}_0 \left[ N_{T_{\ell_1 \rho}} (k)\right]\,.
\end{equation}
Above,   given a generic integer $n\geq 0$,  $N_n (k)$ denotes  the time spent  at $k$ by the random walk  up to time $n$, i.e. 
$N_n (k) =\sum_{r =0}^{n} \mathds{1}( X^{\rho}_r =k) $.

Due to   \cite[Eq.~(50)]{FGS} we have $E[ E_0^{\o,\zeta, \rho }[T_{\ell_1\rho}]]  \geq \e E_0^{ \o, \rho} [ T_\rho]$, thus implying that 
\begin{equation}\label{terracina}
\bbE E[ E_0^{\o,\zeta, \rho }[T_{\ell_1\rho}]]  \geq \e \bbE \bigl[E_0^{ \o, \rho} [ T_\rho] \bigr]\,. 
\end{equation} 
\rosso{We set 
\begin{equation}
 K_1 (\rho, \l):=   \frac{C_1\e}{\l} -\frac{C_2\e}{\rho \l^2}\,.
\end{equation}
Then, }
by combining  Proposition \ref{prop_pinolo}  with \eqref{umbria} and \eqref{terracina},  \rosso{when $K_1(\rho, \l)>0$ we have }
\[ \frac{ d\bbQ^\rho}{d \bbP} (\o) \leq \frac{ 1}{ K_1 (\rho, \l)  \rho }\sum _{k \in \bbZ} E  E^{\t_{-k} \o, \z, \rho}_0 \left[ N_{T_{\ell_1 \rho}} (k)\right]\,.
\]
The above estimate can be rewritten as
\begin{equation}\label{sardegna}
\frac{ d\bbQ^\rho}{d \bbP} (\o) \leq  \frac{H_+(\o)+H_-(\o)  }{ K_1 (\rho,\l) \rho } \,,
\end{equation}
where (as in  \cite[Eq.~(67)]{FGS}) we have defined 
\begin{align*}
 H_+ (\o):=  \sum _{k > 0 } E  E^{\t_{-k} \o, \z, \rho}_0 \Big[ N_{T_{\ell_1 \rho}} (k)\Big]\,,
\qquad
 H_-(\o):=  \sum _{k \leq 0 } E  E^{\t_{-k} \o, \z, \rho}_0 \Big[ N_{T_{\ell_1 \rho}} (k)\Big]\,.
\end{align*}
Note that \eqref{sardegna} equals \cite[Eq.~(67)]{FGS} with the only difference that the constant $K_1$ in \cite{FGS} is now replaced by 
$K_1(\rho, \l)$. The computations done in the proof of Prop.~5.4 in \cite{FGS} show how to go from  \cite[Eq.~(67)]{FGS} to  \cite[Eq.~(77)]{FGS}
by bounding   
$H_+(\o) $ and $H_-(\o)$, and these bounds do not  involve the constant $K_1$ there.  In particular, due to \eqref{sardegna}, the first line in \cite[Eq.~(77)]{FGS} remains valid with $K_1$ replaced with $K_1(\rho, \l) $. In conclusion,  since  the function $g(\o, \l)$ introduced in  \eqref{def_g}  equals the function $g_\o(0)$ defined in \cite[Prop.~3.11]{FGS}, we have:
\begin{align}\label{cetriolo1}
\begin{split}
 \frac{ d\bbQ^\rho}{d \bbP} (\o) & \leq    G_{\rho, \l}(\o):=\frac{C'}{K_1(\rho, \l)}\Big(
\frac{ \pi^1(0) \sum_{ k \leq 0} {\rm e}^{- 2 \l x_{-k}}
 F_*(\t_{-k} \o) }{ \rho}+ g(\o,\l)\Big)\\
 \end{split}
 \end{align}
 \rosso{where the notation has the following meaning. As in \cite{FGS} $\pi^1(0):= c_{-1,0}+ c_{0,1}$ (recall that $\l $ is understood and that in this \rosso{section}  we write $c_{i,j}$ instead of $c_{i,j}^\l$).  $C'$  is a constant depending only on $\e$. 
 Finally,   $F_*$ is the function defined in \cite[Lemma 5.5]{FGS}, i.e.}
 \begin{align*}
 &  F_* (\o):=K_0  \sum_{i=0}^\infty(i+1) {\rm e} ^{-2\l x_{i} +(1-\l) ( x_{i+1} -x_{i})} \,.
 \end{align*}
 Note that 
 the  positive  constant $K_0$ is   independent of $\l \in (0,\l_*]$ and $\rho$ (see \cite[Rem.~3.2]{FGS} and Remark \ref{ricorda}).
 We have that $\lim _{\rho \to \infty} K_1(\rho, \l) \rho=\infty$ and $\lim _{\rho \to \infty} K_1(\rho, \l)=C_1\e/\l$ . Hence, 
 \rosso{for any $\rho \geq \rho_0$ (the latter can depend on $\l$) it holds $K_1(\rho, \l) >0$ and }
\begin{align}
& 0\leq  G_{\rho, \l}(\o) \leq C _3 \Big( \pi^1(0) \sum_{ k \leq 0} {\rm e}^{- 2 \l x_{-k}} F_*(\t_{-k} \o)  + g(\o,\l)\Big)\,, \label{merc1}\\
&\lim _{\rho \to \infty} G_{\rho, \l } (\o) = C_4 \l \, g(\o,\l)\,,\label{merc2}
\end{align}
for suitable positive constants $C_3,C_4$ independent of $\rho, \l$. \rosso{We claim that  the r.h.s. of  \eqref{merc1} is in $L^1(\bbP)$. Indeed, $\pi^1 (0)$ is bounded by an universal constant. The series  appearing in \eqref{merc1}  can be bounded from above  by using the equivalent expression given by   \cite[Eq.~(78)]{FGS}   together with   the property $|x_k| \geq kd $. In this way  one easily gets that the series is in $L^1(\bbP)$. Finally,}   $g(\o,\l)   \in L^1(\bbP)$ \rosso{ due to \cite[Lemma 3.12]{FGS}. By  the above claim, }\eqref{merc1}, \eqref{merc2} and the dominated convergence theorem, we conclude that 
  $G_{\rho,\l}(\o) $ converges to $C_4 \l g(\o,\l)$ in $L^1(\bbP)$.
 Take now   a  bounded positive continuous function $h$ on $\O$.  Since $\bbQ^\rho$ weakly converges to $\bbQ^\infty=\bbQ_\l$  as $\rho \to \infty$ \rosso{(cf.~\cite[Prop.~5.3]{FGS})},  by \eqref{cetriolo1} and the above observations we get
\begin{equation*}
 \bbE \Bigl[ \frac{ d\bbQ^\infty}{d\bbP}  h \Bigr] 
 	= \bbQ^\infty [h]
	= \lim _{\rho \to \infty}  \bbQ^\rho [h] 
	= \lim _{\rho \to \infty}   \bbE \Bigl[ \frac{ d\bbQ^\rho}{d\bbP}  h \Bigr] 
	\leq \lim _{\rho\to \infty} \bbE [ G_{\rho, \l}(\o) h] 
	=  \bbE[ C_4 \l  g(\o,\l) h]\,.
\end{equation*}
The above bound trivially implies \eqref{agognato}.

\section{Proof of Theorem \ref{teo_continuo}}\label{dim_continuo}

\begin{Warning} In the previous section, in order to make more transparent the comparison with the formulas in \cite{FGS}, we used the convention to omit $\l$  \rosso{from the index of several objects}. From now on we  drop this convention and we come back to the notation introduced in Sections \ref{sec_mod_res} and \ref{sec_dim1}. \end{Warning}
\smallskip

Take $f\in L^q(\bbQ_0)$, $p$ and $q$ be as in Theorem \ref{teo_continuo}. The fact that  $f \in L^1(\bbQ_\l)$ is a  simple  consequence of  the H\"older inequality and Theorem \ref{serio_bis}. Indeed we can bound
\[\bbQ_\l(|f| ) =\bbQ_0 \bigl( |f| \frac{ d \bbQ_\l}{ d \bbQ_0} \bigr)\leq \| f\|_{L^q(\bbQ_0)} \| \frac{ d \bbQ_\l}{ d \bbQ_0} \|_{L^p(\bbQ_0)}<\infty\,.\]

The proof of the continuity of the map $ \l \mapsto \bbQ_\l (f) $ is more subtle and uses two main tools. One tool comes from functional analysis and is given by the following proposition (we postpone the proof to Appendix \ref{appendino}):

\begin{Lemma}\label{pizza}  Let $I$ be a finite interval of the real line and let $\l_0 \in I$. Let $Q_\l$, $\l \in I$,  be probability measures on  some measurable space $( \Theta, \cF)$.  Let $L_\l$, $\l \in I$, be a family of  operators  defined on a common subset $\cC$ of $L^2(Q_{\l_0})$, i.e.  $ L_\l:  \cC    \subset L^2(Q_{\l_0})
\to L^2(Q_{\l_0})$.  We assume the following hypotheses:

\begin{itemize}
\item[(H1)] $ Q_\l \ll Q_{\l_0}$ and   $\sup _{ \l \in I  } \| \rho _\l \| _{L^2(Q_{\l_0})} < \infty$, where   $\rho_\l := \frac{ d Q_\l}{\,d Q_{\l_0}}$;
\item[(H2)]   if $Q$ is a probability measure on $(\O, \cF)$ such that $Q \ll Q_{\l_0}$, 
 \rosso{$  \frac{dQ}{d Q_{\l_0}} \in L^2(Q_{\l_0})$} and $ Q( L_{\l_0} f) =0$ for all $ f \in \cC$, then $Q=Q_{\l_0}$;
\item[(H3)] $ Q_\l ( L_\l f)=0$ for all $ \l \in I$ and $f \in \cC$;
\item[(H4)]  $ \lim _{ \l \to \l_0  } \| L_\l f - L_{\l_0} f \| _{L^2(Q_{\l_0})}=0$ for all $f \in \cC$.

\end{itemize}
Then $\rho_\l$ converges to $\rho_{\l_0}$ in the weak topology of $L^2(Q_{\l_0})$, and 
\begin{equation}\label{lego_city}
\lim_{\l \to\l_0}
Q_\l(f)  =Q_{\l_0}(f) \,, \qquad \forall f \in L^2(Q_{\l_0})\,.
\end{equation}
\end{Lemma}

We point out that, in the above lemma, 
 $f\in L^1( Q_\l)$  if $f\in L^2(Q_{\l_0})$, hence the expectation $Q_\l(f)  $ in the l.h.s. of \eqref{lego_city} is well--defined. Indeed, since $ \frac{ d Q_\l}{\,d Q_{\l_0}} \in L^2(Q_{\l_0})$,  it is enough to apply  the Cauchy--Schwarz inequality.

In order to apply the above lemma with $\l_0 \in [0,1)$,    $I:=[\l_0-\d, \l_0+ \d] \subset (0,1) $, $\Theta: = \O$ and $ Q_\l := \bbQ_\l$   to get the continuity of the map $ \l \mapsto \bbQ_\l (f) $  at $\l_0$, we need 
an   upper bound of the norm  $\| \frac{d \bbQ_\l}{\;d \bbQ_{\l_0} }\| _{L^2(\bbQ_{\l_0}) }$ uniformly  in $\l $  as $\l $ varies  in a neighborhood of $\l_0$ (the above mentioned second tool). In the special case $\l_0=0$ this uniform upper bound is provided by Theorem \ref{serio_bis}. For $\l_0 >0$, this bound is stated in the following lemma:

\begin{Lemma}\label{ho_fame}
Suppose that $\bbE[{\rm e}^{2Z_0}]<\infty$. Fix $\l_0 \in (0,1)$ and  $\d>0$ such that $[\l_0-\d, \l_0+ \d] \subset (0,1)$. Then we have
\begin{equation}\label{salsiccia}
\sup_{ \l:\, | \l-\l_0| \leq \d} \Big{\|} \frac{d \bbQ_\l}{\;d \bbQ_{\l_0} }\Big{\|} _{L^2(\bbQ_{\l_0}) } < \infty \,.
\end{equation}
\end{Lemma}
\begin{proof}
In what follows,  we restrict to $\l\in  [ \l_0-\d, \l_0+ \d]$.
We recall that all  $\bbQ_\l$'s are mutually absolutely continuous w.r.t.~$\bbP$ \cite[Thm.~2]{FGS}. As a consequence,   $\bbQ_\l \ll \bbQ_{\l_0}$ and moreover we 
can write
\begin{equation}\label{figlio}
\begin{split}
\Big{\|} \frac{d \bbQ_\l}{\;d \bbQ_{\l_0} }\Big{\|}^2 _{L^2(\bbQ_{\l_0}) } & =
\bbQ_{\l_0} \Big[
\frac{d \bbQ_\l}{\;d \bbQ_{\l_0} }\frac{d \bbQ_\l}{\;d \bbQ_{\l_0} }
\Big]=
\bbQ_{\l} \Big[
\frac{d \bbQ_\l}{\;d \bbQ_{\l_0} }
\Big]= 
 \bbQ_{\l} \Big[
 \frac{d \bbQ_\l}{d \bbP} \Big(  \frac{d \bbQ_{\l_0}}{d \bbP}\Big)^{-1}
\Big]\\
& =  \bbE \Big[ \Big(
 \frac{d \bbQ_\l}{d \bbP}\Big)^2 \Big( \frac{d \bbQ_{\l_0}}{d \bbP}\Big)^{-1}
\Big]\,.
\end{split}
\end{equation}
Due to \eqref{agognato} and assumption (A4)  we can bound $\frac{d \bbQ_\l}{d \bbP} \leq 2 K_0 \sum _{j=0}^\infty {\rm e}^{ - c d j +   Z_j} $ for suitable positive constants $K_0$ and $c$ depending only on  $\l_0$ and $\d$ \rosso{(note that $c_{-1,0}^\l , c_{0,1}^\l$ are bounded by a universal constant from above)}.  On the other hand   \cite[Thm.~2]{FGS} provides the bound    $\frac{d \bbQ_{\l_0}}{d \bbP}\geq \g $, for some strictly positive constant $\g$ depending on $\l_0$. 
By combining the above bounds with \eqref{figlio}, to get 
 \eqref{salsiccia} it is enough to prove that 
$
  \bbE  \Big[  (\sum _{j=0}^\infty {\rm e}^{ - c d j +  Z_j} )^2]<\infty
$.  By expanding the square, the last estimate can be easily checked  since  $\bbE[ {\rm e}^{2 Z_0}] < \infty$. 
\end{proof}

The next step is then to apply Lemma \ref{pizza}  (with the support of Theorem \ref{serio_bis} and Lemma \ref{ho_fame}) to get the continuity of the map $\l \mapsto \bbQ_\l (f)$ for $f \in L^2(\bbQ_0)$. 
To this aim, given a  bounded Borel function $f$ on $\O$, we define $\bbL_\l f$ as
\begin{equation}\label{bagel}
\bbL_\l f (\o) = \sum 
_{k \in \bbZ} p^\l_{0,k}(\o) \bigl[ f(\t_k \o)- f(\o)\bigr]\,.
\end{equation}
Trivially, $\bbL_\l f \in L^2 (\bbQ_\l)$.
We now  consider   Lemma  \ref{pizza} with  $\Theta:= \O$,  $Q_\l:=\bbQ_\l$,  $I:=[\l_0-\d, \l_0+ \d] \subset (0,1) $,  $\cC$ being the set of Borel bounded functions on $\O$ and with  $L_\l$  defined as the above operator $\bbL_\l$ restricted to $\cC$. As an application 
   we   get:
\begin{Lemma} \label{party}
Suppose that $\bbE[{\rm e}^{2Z_0}]<\infty$. Then \rosso{for any  bounded  measurable function $f: \O \to \bbR $}  and for any  $\l_0 \in [0,1)$, it  holds  \begin{equation}\label{lego_friends}
\lim _{\l \to \l_0} \bbQ_\l(f) =\bbQ_{\l_0}(f)\,.
\end{equation}
\end{Lemma}

\begin{proof} \rosso{Since bounded measurable functions are in $L^2(\bbQ_{\l_0})$,
due} to \eqref{lego_city}, to get \eqref{lego_friends} we only need 
to check the hypotheses of  Lemma \ref{pizza} with  $\Theta$, $Q_\l$, $I$, $\cC$ and $L_\l$ defined as above.

\smallskip

Hypothesis (H1) is satisfied due to Theorem \ref{serio_bis} and Lemma \ref{ho_fame}.
Let us check (H2).  Suppose that  $\bbQ$ is a probability on the environment space $\O$  satisfying the properties listed in (H2). Since $\cC$ is dense in $L^2(\bbQ_{\l_0})$ and $\bbQ( \bbL_{\l_0} f)=0$ for any $f \in \cC$, $\bbQ$ is an invariant distribution for the process $(\o^{\l_0}_n)$, defined as   $\o^{\l_0}_n := \t_k\o$ where \rosso{$k= \psi (Y^{\l_0} _n)$}
(the environment  viewed from the walker).   We now want to use that $\bbQ \ll \bbQ_{\l_0}$ to deduce that $\bbQ= \bbQ_{\l_0}$.
To this aim 
we denote by  $\bbP^{\l_0}_\nu$ the law of the process $(\o^{\l_0}_n)$  starting with distribution $\nu $ and by 
$\bbE^{\l_0}_{\nu}$ the associated  expectation. If $\nu=
\d_\o$ we simply write $\bbP^{\l_0}_\o$ and $\bbE^{\l_0}_{\o}$.
We take  $f:\,\O\to\R$ to be any bounded measurable function. By the invariance of $\bbQ$ we have
\begin{align}\label{raffineria}
\bbQ[f]=\bbE^{\l_0}_{\bbQ}\Big[\frac 1n\sum_{j=0}^{n-1}f(\omega_j^{\l_0})\Big]=\bbQ[F_n]\,,
\end{align}
where $F_n(\o):=\bbE^{\l_0}_{\o}\Big[\frac 1n\sum_{j=0}^{n-1}f(\omega_j^{\l_0})\big]$. 
Now, since $\bbQ_{\l_0}$ is ergodic, we know that for
\begin{align*}
A:=\Big\{\o\in\Omega:\,
\lim_{n\to\infty} \frac1n\sum_{j=0}^{n-1}f(\o_j^{\l_0})=\bbQ_{\l_0}[f]\quad \P_\o^{\l_0}-a.s.\Big\}
\end{align*}
we have $\bbQ_{\l_0}[A]=1$.
Since the map $(\o_j^{\l_0})_{j\geq 0}\to\frac{1}{n}\sum_{j=0}^{n-1}f(\o_j^{\l_0})$ is bounded by $\|f\|_\infty$, we can apply the dominated convergence theorem to obtain that, for each $\o\in A$,
$$\lim_{n\to\infty}F_n(\o)=\bbQ_{\l_0}[f]\,.$$
To conclude, we would like to apply  again the dominated convergence theorem to analyze $\lim_{n\to\infty}\bbQ[F_n]$. We can do that since $|F_n(\o)|\leq\| f \|_{\infty}$ and  since $F_n(\o)\to\bbQ_{\l_0}[f]$  for $\bbQ$--a.a. $\o$ (because $\bbQ\ll\bbQ_{\l_0}$ and $\bbQ_{\l_0}(A)=1$, thus implying that $\bbQ(A)=1$). We then  obtain that 
$
\lim_{n\to\infty}\bbQ[F_n]=\bbQ_{\l_0}[f].
$
By \eqref{raffineria} we get $\bbQ[f]=\bbQ_{\l_0}[f]$. Since this is true for every $f$, we have $\bbQ=\bbQ_{\l_0}$.

 \smallskip

(H3) follows from the fact that $\bbQ_\l$ is an invariant distribution  for the process  ``environment  viewed from the random walk $Y^{\l}_n$''.  
 
 \smallskip
 
It remains to check (H4). Since $f \in \cC$ is bounded, it is enough to have
\begin{equation}\label{seianni!}
 \lim _{ \l \to \l_0} \bbQ_{\l_0} \Bigl[ \Big( \sum _{k\in\bbZ} | p_{0,k}^\l - p_{0,k}^{\l_0} |\Big)^2 \Big] 
 	=0\,.
\end{equation}
To conclude we observe that, by writing  $\bbQ_{\l_0} [\cdot]= \bbQ_0 \big[  \frac{d \bbQ_{\l_0}  }{ d \bbQ_0} \cdot\big]$,   \eqref{seianni!} follows from 
the  Cauchy--Schwarz inequality, the fact that  $ \frac{d \bbQ_{\l_0}  }{ d \bbQ_0} \in L^2(\bbQ_0)$ and Lemma \ref{tuttabirra} in Appendix \ref{poux}.
\end{proof}

As a byproduct of Theorem \ref{serio_bis}, Lemma \ref{ho_fame} and Lemma \ref{party} we can complete the proof of Theorem \ref{teo_continuo}. 
To this aim we suppose the assumptions of Theorem \ref{teo_continuo} to be satisfied and we   take $f \in L^q(\bbQ_0)$ and $\l_0 \in [0,1)$. We  \rosso{ take  $\l_*\in (\l_0, 1)$ and from now on we restrict to $\l \in [0,\l_*]$}.  Recall that at the beginning of this section we have proved that $f \in L^1(\bbQ_\l)$.

We want to show that $\bbQ_\l(f) \to \bbQ_{\l_0} (f)$ as $\l \to \l_0$.  To this aim, given $M>0$,  we define $f_M(\o) $ as $M$ if $f(\o)>M$, as $-M$ if $f(\o) <-M$ and as $f(\o)$ otherwise. We then can  bound
\begin{equation}\label{napoletana}
\begin{split}
| \bbQ_\l(f)- \bbQ_{\l_0}(f)|  &   \leq | \bbQ_\l(f)- \bbQ_\l(f_M)|  + | \bbQ_\l(f_M)- \bbQ_{\l_0}(f_M)|  \\
& + | \bbQ_{\l_0}(f_M)- \bbQ_{\l_0}(f)| \,.
\end{split}
\end{equation}
To conclude  it is enough to show that the r.h.s.~of \eqref{napoletana} goes to zero when we take first the limit $\l \to \l_0$ and afterwards the limit $M\to \infty$. \rosso{Due to  Lemma \ref{party} the second  term in the   r.h.s.~of \eqref{napoletana} goes to zero already as $\l\to \l_0$ since  $f_M$ is bounded. On the other hand, by the H\"older inequality,  the first  and 
third terms in  the   r.h.s.~of \eqref{napoletana}  can be  bounded by 
\[ \| f-f_M\|_{L^q(\bbQ_0)} \sup_{\z \in [0,\l_*] } \Big \| \frac{ d \bbQ_\z}{ d \bbQ_0}\Big \|_{L^p(\bbQ_{0})}\,.
\]
Note the independence from $\l$ of the above expression. Since $f\in L^q(\bbQ_0)$, $ \| f-f_M\|_{L^q(\bbQ_0)}$ goes to zero as $M\to \infty$ by the dominated convergence theorem, thus completing the proof.
}


\section{Proof of Theorem \ref{teo_derivo} (first part)}\label{dim_teo_derivo}
In this section we prove 
  the existence of  $\partial_{\l=0} \bbQ_\l (f)$ and equation \eqref{sole}.
 As in the theorem, we suppose that $\bbE[{\rm e}^{p Z_0}] <\infty$  for some $p>2$ and that  \rosso{$f \in H_{-1}\cap L^2(\bbQ_0)$}. In what follows, $q$ is the exponent conjugate to $p$, i.e.~the value satisfying  ${p}^{-1}+{q}^{-1}=1$.

 To simplify the notation 
we write here $g_\e$, $h$  instead of the functions $g_\e^f$,  $h^f$  introduced in \eqref{def_gigi}, \eqref{hf}, respectively.  Recall that, 
given $\e>0$, 
$ g_{\e }\in L^2( \bbQ_0)  $ is  the solution  of the equation
$\e g_\e -\bbL_0 g_\e = f$. Since   $L^2 (\bbQ_0) \subset  L^1 (\bbQ_\l)$ \rosso{(by Theorem \ref{serio_bis} and the Cauchy--Schwarz inequality)}, the above identity on $g_\e$  implies that $ \bbQ_\l (f)=     \e \bbQ_\l (g_\e)-  \bbQ_\l ( \bbL_0 g_\e)$. Using that $\bbQ_0(f)=0$ since $f\in H_{-1}$, we can write
\begin{equation}\label{terramare}
\frac{ \bbQ_\l (f)-\bbQ_0(f)}{\l}=    \frac{ \e \bbQ_\l (g_\e)}{\l}- \frac{ \bbQ_\l ( \bbL_0 g_\e)}{\l} \,.
 \end{equation}
 In what follows we will take first the limit $\e\to 0$ and afterwards the limit $\l \to 0$.
 
Since $f \in H_{-1}$ we can apply the results and estimates of \cite{KV}. In particular, it holds $\e \| g_\e \|^2 _{L^2(\bbQ_0)}  \to 0$ as $\e \to 0$ (see \cite[Eq.~(1.12)]{KV}) and, due to Theorem \ref{serio_bis}, we can bound 
\begin{equation}\label{5terre2}
\bigl |\e   \bbQ_\l (g_\e) \bigr|
	=\Bigl |\e  \big\langle \frac{d\bbQ_\l}{d\bbQ_0}, g_\e\big\rangle  \Bigr |    
	\leq \e \| g_\e\|_{L^2(\bbQ_0)}   \Big\|  \frac{d \bbQ_\l}{d \bbQ_0} \Big \|_{L^2(\bbQ_0)} 
 \to  0 \text{ as } \e \to 0\,.
\end{equation}
\rosso{We recall that    the scalar product  in $L ^2(\bbQ_0)$   is denoted   by $\langle \cdot, \cdot \rangle$. 
As a consequence of \eqref{5terre2}},  the first  term in the r.h.s. of \eqref{terramare} is negligible as $\e \to 0$.

\medskip
It remains to analyze the second term in  the r.h.s. of \eqref{terramare}.
 Recall the space $L^2(M)$  of square integrable forms introduced in Section \ref{sec_mod_res} and recall \eqref{cirm}.


\begin{Lemma}\label{uffina}
Let  $\bbE[{\rm e}^{p Z_0}] <\infty$  for some $p>2$.
Let $\hat q >2$ be such that  $\frac{1}{p}+ \frac{1}{\hat  q} =\frac{1}{2}$.
Given a  form $v $ with $v( \cdot, 0)\equiv 0$  and  a square integrable form $w \in L^2(M)$, there exists $C>0$ such that  for all  $\l \in (0,1/2)$ it holds 
\begin{equation}
\bbQ_\l\Big[ \sum_{k\in \bbZ } | v(\cdot,k) w(\cdot , k) | \Bigr]\leq C\,
 \|w\|_{L^2(M)}
 \bbQ_0 \Bigl[  \sum_{k\in \bbZ \setminus\{0\}  }   p_{0,k} \Big|\frac{ v (\cdot,k)}{p_{0,k} } \Big|^{\hat  q}\Bigr]^\frac{1}{\hat  q }\,.
\end{equation} 
\end{Lemma}
\begin{proof}
We simply compute 
\begin{align*}
\bbQ_\l\Big[ \sum_{k\in \bbZ    } | v(\cdot,k) w(\cdot , k) | \Bigr]
	&=\bbQ_0\Big[ \sum_{k\in \bbZ\setminus\{0\} } p_{0,k}\Big| \frac{d\bbQ_\l}{d\bbQ_0}\frac{v(\cdot,k )}{p_{0,k}}\Big|\,\bigl| w(\cdot , k)\bigr | \Bigr] \\
	&\leq \|w\|_{L^2(M)} \bbQ_0\Big[ \sum_{k\in \bbZ \setminus\{0\}  } p_{0,k}\Big( \frac{d\bbQ_\l}{d\bbQ_0}\Big)^2\Big(\frac{v(\cdot, k)}{p_{0,k}}\Big)^2\Bigr]^{1/2} \\
	&\leq \|w\|_{L^2(M)}
 \bbQ_0 \Big[ \Big( \frac{ d \bbQ_\l}{d\bbQ_0} \Big)^p \Big]^\frac{1}{p}\bbQ_0 \Bigl[  \sum_{k\in \bbZ\setminus\{0\} }  p_{0,k} \Big|\frac{ v(\cdot, k)}{p_{0,k}} \Big|^{\hat  q}\Bigr]^\frac{1}{\hat  q },
\end{align*}
where for the second line we have used the  Cauchy-Schwarz inequality with respect to  the measure $M$, while for the second inequality we used the H\"older inequality again with respect to   $M$ and with exponents $p/2$ and $\hat q/2$, so that $(p/2)^{-1}+(\hat q/2)^{-1}=1$ by hypothesis. We also have used the fact that $M[( \frac{ d \bbQ_\l}{d\bbQ_0} )^p]=\Q_0[( \frac{ d \bbQ_\l}{d\bbQ_0} )^p]$. To conclude it is enough to apply Theorem \ref{serio_bis}.
\end{proof}
  \begin{Lemma}\label{andiamo} 
  Let  $\bbE[{\rm e}^{p Z_0}] <\infty$  for some $p>2$ and let $\hat{q}$ be as in Lemma \ref{uffina}.
 Then there exists a constant $C$ not depending on $\l \in [0,\frac 1{2\hat q})$ such that, for any form $w\in L^2(M)$, it holds 
  \[
   \bbQ_\l  \Big[ \sum _{ k\in \bbZ  } \big| (p^\l_{0,k} -  p_{0,k} ) w(\cdot,k )\big| \Big]\leq C \| w\| _{L^2(M)} \,.\]
\end{Lemma}
  \begin{proof}
 In this proof the  constants $C,C'$ are positive,  might vary from line to line and do not depend on the specific choice of $\l\in [0,\frac 1{2\hat q})$.
By applying Lemma  \ref{uffina} with $v(\cdot,k)=p^\l_{0,k} -  p_{0,k}$  we already know that
 \begin{align}\label{shakespeare}
   \bbQ_\l  \Big[ \sum _{ k\in \bbZ } \big| (p^\l_{0,k} -  p_{0,k} ) w(\cdot,k )\big| \Big]
   	\leq  C \| w\| _{\rosso{L^2(M)}}
	 \bbQ_0\Bigl[\sum_{k\in \bbZ \setminus\{0\}}  p_{0,k} \Big|\frac{ p^\l_{0,k}}{p_{0,k}}-1\Big|^{\hat q}\Bigr]^\frac{1}{\hat q}\,.
\end{align} 
Since for $a \geq 0$ it holds $|a-1|^q \leq |a|^q +1$, we can bound
\begin{equation}\label{ascensione}
\bbQ_0 \Bigl[  \sum_{k\in \bbZ \setminus\{0\}}  p_{0,k} \Big|\frac{ p^\l_{0,k}  }{p_{0,k}}-1 \Big|^{\hat  q}\Bigr]
	\leq 1+\bbQ_0 \Bigl[  \sum_{k\in \bbZ \setminus\{0\}}  p_{0,k} \Big|\frac{ p^\l_{0,k}  }{p_{0,k}}\Big|^{\hat  q}\Bigr]\,.
\end{equation}
Since $\frac{ p^\l_{0,k}  }{p_{0,k}}\leq C{\rm e}^{\l (x_k+ Z_{-1})}$ (see \eqref{lega} in the Appendix for a proof of this fact), we can bound
\begin{align}\label{spagna}
\bbQ_0 \Bigl[  \sum_{k\in\bbZ \setminus\{0\}}  p_{0,k} \Big|\frac{ p^\l_{0,k}  }{p_{0,k}}\Big|^{\hat  q}\Bigr]
	&\leq C\,\bbQ_0 \Bigl[  \sum_{k\in\bbZ}  p_{0,k}  {\rm e}^{\l \hat q(|x_k|+ Z_{-1})}\Bigr]
	\leq C \rosso{{\rm e}^{\| u\|_\infty}}\,\bbE \Bigl[ \frac{d\bbQ_0}{d\P} \sum_{k\in\bbZ}\frac{ {\rm e}^{- |x_k|}}{\pi} {\rm e}^{\l \hat q(|x_k|+ Z_{-1})}\Bigr]\nonumber\\
	&\leq C'\,\Big(\sum_{k\in\bbZ} {\rm e}^{-(1-\l\hat q) |k|d}\Big)\,\bbE \Bigl[{\rm e}^{\l \hat q Z_{-1}}\Bigr] \,,
\end{align}
where for the last inequality we have used that $\frac{d\bbQ_0}{d\P}=\frac{\pi}{\bbE[\pi]}$ and that $|x_k|\geq |k|d$. Note that the last line in \eqref{spagna} is uniformly bounded for $\l\in(0,\frac{1}{2\hat q})$ (recall that $\bbE[ {\rm e}^{pZ_0}] <\infty$). This bound together with \eqref{shakespeare} and \eqref{ascensione} allows to conclude.
 \end{proof}

   \begin{Lemma}
    Given  $g \in L^2(\bbQ_0)$, the series $\sum _{k \in \bbZ} p_{0,k}^\l | g(\t_k \cdot)- g(\cdot)|$ belongs to $L^1(\bbQ_\l)$. Defining,  as in \eqref{bagel}, $\bbL_\l g(\o) :=\sum _{k \in \bbZ} p_{0,k}^\l ( g(\t_k \cdot)- g(\cdot))$, we get that $\bbL_\l g \in L^1(\bbQ_\l)$ and $\bbQ_\l ( \bbL_\l g) =0$.
   \end{Lemma}
   \begin{proof} Recall that  $\bbQ_\l$ is an invariant distribution for the  environment viewed from the perturbed  walker, i.e.~ for    $( \t_{Y^\l_n} \o)_{n\geq 0}$. This implies that  $\bbQ_\l\bigl[ \sum _{k\in \bbZ} p_{0,k}^\l  |g( \t_k\cdot) |\bigr] = \bbQ_\l [|g|] <\infty $ (in the last bound we have used  Theorem \ref{serio_bis} to get $g \in L^1(\bbQ_\l)$). As a consequence,  $\sum _{k \in \bbZ} p_{0,k}^\l | g(\t_k \cdot)- g(\cdot)|$ belongs to $L^1(\bbQ_\l)$ and therefore $\bbL_\l g$ is a well--defined element of $L^1(\bbQ_\l)$.
  Finally, again by the invariance of $\bbQ_\l$, we have 
   $\bbQ_\l [g] = \bbQ_\l\bigl[ \sum _{k\in \bbZ} p_{0,k}^\l g( \t_k\cdot) \bigr]$, which is equivalent to $\bbQ_\l(\bbL_\l g) =0$.
 \end{proof}

By the above lemma ${\bbQ_\l} ( \bbL_\l g_\e)$ is well--defined and equals zero. Hence we can write
\begin{equation}\label{russia} -  {\bbQ_\l} ( \bbL_0 g_\e) =  
  {\bbQ_\l}([\bbL_\l- \bbL_0] g_\e)=  {\bbQ_\l} \Big[     \sum _{k \in \bbZ } ( p^\l_{0,k} -  p_{0,k} )( g_\e (\t_k \cdot )- g_\e) \Big]  \,.
  \end{equation} 
 
 By \cite[Eq.~(1.11a)]{KV} we have that the sequence  $g_\e $ is Cauchy, as $\e \downarrow 0$,  in the space $H_1 $ referred to the operator $- \bbL_0$. In particular, we have 
\begin{equation}\label{cauchy} \lim _{\e_1,\e_2 \downarrow 0 } \bbQ_0 \Big[ \sum_{k\in \bbZ} p_{0,k} \Bigl(
(g_{\e_1}-g_{\e_2})  (\t_k \cdot ) - (g_{\e_1}-g_{\e_2})  
 \Bigr)^2 \Big]=0\,.
 \end{equation}
 \eqref{cauchy} can be restated as follows: The family of quadratic forms $(\nabla g_\e)_{\e>0}$ is 
Cauchy in  $L^2 ( M)$. As a consequence, we get that $\nabla g_\e \to h$ in 
 $L^2 ( M)$ for some form  $h\in  L^2 (M)$.  
  Finally, we point out that,  due to Lemma \ref{andiamo}, the expectation   \rosso{$\bbQ_\l  [ \sum _{ k \in \bbZ} ( p^\l_k -  p^0_k )h (\cdot,k) ]$} is well--defined. 
\begin{Lemma}\label{gita} It holds
 \begin{equation}
\lim _{\e \downarrow 0} \Bigl|  {\bbQ_\l} ( \bbL_0 g_\e) +\bbQ_\l  \Bigl [ \sum _{ k \in \bbZ }\bigl ( p^\l_{0,k} -  p_{0_k}\bigr)h(\cdot,k) \Bigr]\Bigr|=0\,.
 \end{equation} 
 \end{Lemma}
 \begin{proof} 
 We set \rosso{$w_\e =\nabla   g_\e -h $}. Due to \eqref{russia}   we only need to show that 
  \begin{equation}\label{mosca}
 \lim _{\e \downarrow  0}   \bbQ_\l   \Big[    \sum _{k  \in \bbZ } \big| ( p^\l_k -  p^0_k)w_\e(\cdot, k) \big| \Big]=0\,.
   \end{equation}
   By applying Lemma \ref{andiamo}  and  using that $\lim _{\e\to 0}\nabla g_\e =h$ in $L^2(M)$,  we get the claim.
%
\end{proof}


\begin{Lemma}\label{portamivia}
It holds
\begin{equation}\label{5terre3}
\begin{split}
 \lim _{\l \downarrow 0}   \frac{1}{\l} \bbQ_\l  \Big[ \sum _{ k\in \bbZ }\Big ( p^\l_{0,k} -  p_{0,k}\Big)h(\cdot,k)  \Big]\ =  {\bbQ_0}
\Big[ \sum _{k\in \bbZ}    \partial_{\l=0}  p^\l_{0,k} h(\cdot,k) \Big] \,.
\end{split}
\end{equation}
\end{Lemma}
\begin{proof} 
We can write
\begin{equation}\label{porta}
\begin{split}
\bbQ_\l  \Big[ \sum _{ k\in \bbZ  } \frac{ p^\l_{0,k} -  p_{0,k}}{\l} h(\cdot,k)  \Big] & = {\bbQ_\l}
\Big[ \sum _{k\in \bbZ}    \partial_{\l=0}  p^\l_{0,k} h(\cdot,k) \Big]\\
& +
 \bbQ_\l  \Big[ \sum _{ k\in \bbZ  } \bigl( \frac{ p^\l_{0,k} -  p_{0,k}}{\l} -\partial_{\l=0} p_{0,k}^\l \bigr) h(\cdot,k)  \Big]\,.
\end{split}
\end{equation}
In the first part  of the proof (Step 1) we show that the first  term in the r.h.s.~ converges to the r.h.s.~of \eqref{5terre3}, while in the second part (Step 2) 
 we show that the second term in the r.h.s.~ goes to zero as $\l \to 0$.

\noindent
{\bf Step 1}. Due to Theorem \ref{teo_continuo} it is enough to show that 
$ \sum _{k\in \bbZ}    \partial_{\l=0}  p^\l_{0,k} h(\cdot,k)$ belongs to $L^q (\bbQ_0)$. Since 
$\partial_{\l} p_{0,k}^\l =p_{0,k}^\l(x_k-\varphi_\l)$,   we can rewrite $ \sum _{k\in \bbZ}    \partial_{\l=0}  p^\l_{0,k} h(\cdot,k)$  as 
$ \sum _{k\in \bbZ}    p_{0,k} (x_k-\varphi) h(\cdot,k)$.
Applying the Cauchy-Schwarz inequality  we get
\begin{align*}
\Big\|\sum _{k\in \bbZ}    p_{0,k} (x_k-\varphi) h(\cdot,k)\Big\|_{L^q(\Q_0)}^q
	&\leq \bbQ_0\Big[ \Big(\sum _{k\in \bbZ}    p_{0,k} (x_k-\varphi)^2\Big)^{q/2} \Big(\sum _{k\in \bbZ}    p_{0,k}h(\cdot,k)^2\Big)^{q/2}\Big].
\end{align*}
We choose now exponents $A:=2/q>1$ and $B:=2/(2-q)$ such that $A^{-1}+B^{-1}=1$ and apply the H\"older inequality to the previous display obtaining
\begin{align*}
\Big\|\sum _{k\in \bbZ}    p_{0,k} (x_k-\varphi) h(\cdot,k)\Big\|^q_{L^q(\Q_0)}
	\leq \bbQ_0\Big[ \Big(\sum _{k\in \bbZ}    p_{0,k} (x_k-\varphi)^2\Big)^{\frac{qB}{2}}\Big]^{1/B}
			 \bbQ_0\Big[ \sum _{k\in \bbZ}    p_{0,k}h(\cdot,k)^2\Big]^{1/A}.
\end{align*}
The second factor in the r.h.s.~is bounded since $h\in L^2(M)$. For finishing Step 1 we are thus left to show that 
\begin{align}\label{rosso}
\bbQ_0\Big[ \Big(\sum _{k\in \bbZ}    p_{0,k} (x_k-\varphi)^2\Big)^{\frac{qB}{2}}\Big]<\infty.
\end{align}
 By the Cauchy--Schwarz inequality one has $\varphi ^2 =(\sum_k p_{0k} x_k)^2 \leq \sum_k p_{0k} x_k^2 $ so that 
\begin{equation}\label{grecia1}
\sum_{k\in\Z^d} p_{0,k} (x_k- \varphi) ^2 \leq 2
\sum _{k\in\Z^d} p_{0,k} x_k^2 +2\sum_{k\in\Z^d} p_{0,k} \varphi^2
\leq 4\, \sum_{k\in\Z^d} p_{0,k} x_k^2 \,.
\end{equation}
Since $qB/2=q/(2-q)>1$, by the H\"older inequality we have 
\begin{equation}\label{grecia2}
\Bigl(  \sum_{k\in\Z^d} p_{0,k} x_k^2 \Bigr)^\frac{qB}{2} 
	\leq \sum_{k\in\Z^d} p_{0,k} x_k^{q B} \,.
\end{equation}
At this point \eqref{rosso} follows from \eqref{grecia1}, \eqref{grecia2} and \eqref{passetto} in Appendix \ref{poux}.

\bigskip
\noindent{\bf Step 2}. By Taylor expansion with the Lagrange rest we can write 
\begin{equation}\label{espansione}
\frac{ p^\l_{0,k} -  p_{0,k}}{\l} -\partial_{\l=0} p_{0,k}^\l 
	=\frac{ \l}{2} \,\partial^2 _{\l=\xi_k} p_{0,k}^\l \,,
\end{equation}
where $\partial^2 _{\l=\xi_k} p_{0,k}^\l $ denotes the second derivative of the function $\l \mapsto  p_{0,k}^\l$  evaluated at some  $\xi_k\in[0,\l]$.
 To prove that the second term in the r.h.s.~of \eqref{porta} is negligible as $\l \to 0$, it is therefore enough to show  that, for some $\d>0$,  
\begin{equation}
\sup _{ \rosso{\l \in [0,\d ]}} \bbQ_\l \Big[ \sum _{ k\in \bbZ  } \big|v(\cdot,k) h(\cdot,k) \big|\Big] <\infty , \qquad  v(\cdot,k)
	:=\sup_{    \rosso{\xi_k\in [0,\d]} } |\partial^2 _{\l=\xi_k}  p_{0,k}^\l|\,.
\end{equation}
By Lemma \ref{uffina}, since $h \in L^2(M)$, it is enough to show 
 \begin{equation*}
  \bbQ_0 \Bigl[  \sum_{k\in \bbZ\setminus\{0\} }   p_{0,k} \Big|\frac{ v (\cdot,k)}{p_{0,k} } \Big|^{\hat  q}\Bigr]^\frac{1}{\hat  q }<\infty
 \end{equation*}
  where 
 $\hat q >2$ is such that  $\frac{1}{p}+ \frac{1}{\hat  q} =\frac{1}{2}$. This follows from 
\eqref{limitato3} in  Lemma \ref{limitato} in Appendix \ref{poux}.
\end{proof}


By collecting \rosso{Lemma \ref{gita}}, Lemma \ref{portamivia}  and using that   $ \partial_{\l=0}  p^\l_{0,k}  =p_{0,k} (x_k -\varphi) $ we obtain 
\begin{equation}
\lim _{ \l \downarrow  0 }\lim_{\e\downarrow 0} 
- \frac{ \bbQ_\l ( \bbL_0 g_\e)}{\l}=  {\bbQ_0}
\Big[ \sum _{k\in \bbZ}    \partial_{\l=0}  p^\l_{0,k} h(\cdot,k) \Big]=  {\bbQ_0}
\Big[ \sum _{k\in \bbZ}  p_{0,k} (x_k -\varphi) h(\cdot,k) \Big]\,.
\end{equation}

This together with \eqref{terramare} and \eqref{5terre2} gives that $\bbQ_\l[f]$ is derivable at $\l=0$ and we obtain \eqref{sole}.

\section{Proof of Theorem \ref{teo_derivo} (second part)}\label{cetriolo}

In this section \rosso{we} deal with the second identity in Theorem \ref{teo_derivo}, that is, equation \eqref{luna}, and show how it can be derived from \eqref{sole}.  Recall the process $(\omega_n)$ of the environment viewed from  the unperturbed walker  $(Y_n)$ defined through $\omega_n=\tau_{Y_n}\omega$, where $\omega$ denotes the initial environment. Below we denote by $\| \cdot \|_{-1}$ the $H_{-1}$ norm referred to the operator $-\bbL_0$ in $L^2(\bbQ_0)$ and by $\langle \cdot, \cdot \rangle$ the scalar product in $L^2(\bbQ_0)$.

\begin{Lemma}\label{miomistero}
For any $V\in H_{-1}\cap L^2(\bbQ_0)$, the sequence $\frac{1}{\sqrt n}\sum_{j=0}^{n-1}V(\omega_j)$ converges weakly as $n\to\infty$ to a Gaussian random variable with variance $\sigma^2=2\|V\|_{-1}^2-\|V\|_{L^2(\bbQ_0)}^2$.
\end{Lemma}
 
\begin{proof}
By \cite[Cor.~1.5]{KV} we have that $\frac{1}{\sqrt n}\sum_{j=0}^{n-1}V(\omega_j)$ converges to a Gaussian random variable with variance given by (see \cite[Eq.~(1.1)]{KV})
\begin{align*}
\s^2=\int_{[0,1]}\frac{1+\theta}{1-\theta}\,\mathfrak{m}_V(d\theta)<\infty\,,
\end{align*}
where $\mathfrak{m}_V$ denotes the spectral measure of $V$ associated to the symmetric  operator $S_0$ on   $L^2(\bbQ_0)$   defined as $S_0 f(\o):= \sum _{k\in \bbZ} p_{0,k} f(\t_k \o)$. Since $-\bbL_0= \bbI-S_0 $, by spectral calculus we  obtain
\[
\s^2=2\int_{[0,1]} \frac{1}{1-\theta}\,\mathfrak{m}_V(d\theta)-\int_{[0,1]} \mathfrak{m}_V(d\theta)	=2\|V\|_{-1}^2-\|V\|_{L^2(\bbQ_0)}^2\,.\qedhere
\]
\end{proof}
 
Let $f\in H_{-1} \cap L^2(\bbQ_0) $ be as in Theorem \ref{teo_derivo}. 
 A direct consequence of the above lemma is that, for the gaussian variables $N^f$ and $N^\varphi$ considered in \eqref{jenner}, it holds $Var(N^f)=2\|f\|_{-1}^2-\|f\|_{L^2(\bbQ_0)}^2$, $Var(N^\varphi)=2\|\varphi\|_{-1}^2-\|\varphi\|_{L^2(\bbQ_0)}^2$ and $Var(N^f+ N^\varphi )=2\|f+\varphi\|_{-1}^2-\|f+\varphi\|_{L^2(\bbQ_0)}^2$.
 By this we obtain a first formula for their covariance:
 \begin{align}\label{magnum}
Cov(N^f,N^\varphi)
	&=\frac{1}{2}\bigl( Var(N^f+N^\varphi)-Var(N^f)-Var(N^\varphi)\bigr)\nonumber\\
	&=\|f+\varphi\|_{-1}^2-\|f\|_{-1}^2-\|\varphi\|_{-1}^2-\langle f,\varphi\rangle.
\end{align}

We are now ready to show \eqref{luna}.  In what follows, we write $g_\e, h$ for the functions $g_\e^f$, \rosso{$h^f$} introduced in \eqref{def_gigi}, \eqref{hf}, respectively.
 Recall by \eqref{sole} that one has 
 \begin{align}\label{soletto}
\partial_{\l=0}\bbQ_\l(f)
	=\bbQ_0 \Bigl[ \sum_{k\in\bbZ} p_{0,k} x_k h(\cdot,k) \Bigr]
		-\bbQ_0 \Bigl[ \varphi\sum_{k\in\bbZ} p_{0,k}  h(\cdot,k) \Bigr].
\end{align}
We divide the proof into the two following claims, that together with \eqref{magnum} and \eqref{soletto}  clearly imply  \eqref{luna}.
\begin{Claim}\label{ballo1} We have
 \begin{align*}
\bbQ_0 \Bigl[ \varphi\sum_{k\in\bbZ} p_{0,k}  h (\cdot,k) \Bigr]=-\langle f,\varphi\rangle\,.
\end{align*}
\end{Claim}

\begin{Claim}\label{ballo2} We have
 \begin{align*}
-\bbQ_0 \Bigl[ \sum_{k\in\bbZ} p_{0,k} x_k h(\cdot,k) \Bigr]
	=\|f+\varphi\|_{-1}^2-\|f\|_{-1}^2-\|\varphi\|_{-1}^2\,.
\end{align*}
\end{Claim}
\begin{proof}[Proof of Claim \ref{ballo1}]
We can  write
\begin{align*}
\bbQ_0 \Bigl[\varphi  \sum_{k\in\bbZ} p_{0,k}  h(\cdot,k) \Bigr]
	&= \bbQ_0 \Bigl[ \varphi  \sum_{k\in\bbZ} p_{0,k} \nabla g_\e(\cdot,k) \Bigr]+
		\bbQ_0 \Bigl[ \varphi \sum_{k\in\bbZ} p_{0,k}  \big(h(\cdot,k)-\nabla g_\e(\cdot,k)\big) \Bigr]
	\,.
\end{align*}
\rosso{We denote by $A_\e$ and $B_\e$ the two terms in the r.h.s.~of the above expression.}
We now show that, as $\e\downarrow 0$, $A_\e\to-\langle f,\varphi\rangle$ and $B_\e\to 0$, which gives the claim.

Since $(\e-\bbL_0)g_\e=f$, we have 
\begin{align*}
A_\e=\bbQ_0[\varphi(\bbL_0g_\e)]=\e\bbQ_0[\varphi  g_\e]-\bbQ_0[\varphi f].
\end{align*}
For the first summand we can bound
\begin{align*}
\big|\e\bbQ_0[\varphi  g_\e]\big|\leq\e\|\varphi\|_{L^2(\bbQ_0)}\|g_\e\|_{L^2(\bbQ_0)}\xrightarrow{\e\downarrow 0}0
\end{align*}
since, by \cite[Eq.~(1.12)]{KV}, we know that $\e\|g_\e\|_{L^2(\bbQ_0)} \to0$ as $\e\downarrow0$. 
This implies $\lim_{\e\downarrow0}A_\e=-\bbQ_0[\varphi f]=-\langle f,\varphi \rangle$.

Turning to $B_\e$, by \eqref{hf} and  the Cauchy--Schwarz inequality with respect to the measure $M$, we have
\begin{align*}
|B_\e|
	&\leq \bbQ_0 \Bigl[ \sum_{k\in\bbZ} p_{0,k} \varphi^2 \Bigr]^{\frac 12}
		\bbQ_0 \Bigl[ \sum_{k\in\bbZ} p_{0,k} (h(\cdot,k)-\nabla g_\e(\cdot,k))^2 \Bigr]^{\frac12}\\
	&=\|\varphi\|_{L^2(\bbQ_0)}\|h-\nabla g_\e\|_{L^2(M)}\xrightarrow{\e\downarrow 0} 0\,.\qedhere
\end{align*}
\end{proof}

\begin{proof}[Proof of Claim \ref{ballo2}]
First of all we notice that \begin{align}\label{essere}
\bbQ_0 \Bigl[ \sum_{k\in\bbZ} p_{0,k} x_k h(\cdot,k) \Bigr]
	=\lim_{\e\downarrow0}\bbQ_0 \Bigl[ \sum_{k\in\bbZ} p_{0,k} x_k \nabla g_\e(\cdot,k) \Bigr].
\end{align}
Indeed, by the Cauchy--Schwarz inequality and \eqref{hf}, it holds 
\begin{align*}
\Big|\bbQ_0 \Bigl[ \sum_{k\in\bbZ} p_{0,k} x_k (h(\cdot,k)-\nabla g_\e(\cdot,k)) \Bigr]\Big|
	&\leq \bbQ_0 \Bigl[ \sum_{k\in\bbZ} p_{0,k} x_k^2\Big]^\frac12
		\bbQ_0 \Bigl[ \sum_{k\in\bbZ} p_{0,k} (h(\cdot,k)-\nabla g_\e(\cdot,k))^2 \Bigr]^\frac 12\\
	&=\bbQ_0 \Bigl[ \sum_{k\in\bbZ} p_{0,k} x_k^2\Big]^\frac12 \|h-\nabla g_\e\|_{L^2(M)}		\xrightarrow{\e\to0} 0\,.
\end{align*}

The expectation in the r.h.s.~of \eqref{essere} can be rewritten as 
\begin{align}\label{bello}
\bbQ_0 \Bigl[ \sum_{k\in\bbZ} p_{0,k} x_k  (g_\e(\tau_k\cdot )- g_\e )\Big]
	=-2\bbQ_0 \Bigl[ \sum_{k\in\bbZ} p_{0,k} x_k  g_\e \Big]
	=-2\bbQ_0[\varphi g_\e].
\end{align}
To see why the first equality holds we just note that for each $k\in\Z$
\begin{align*}
\bbQ_0 \Bigl[  p_{0,k} x_k  g_\e(\tau_k\cdot)\Big]
	&=\frac{1}{\E[\pi]}\bbE[  r_{0,k} x_k  g_\e(\tau_k\cdot )\Big]
	=\frac{1}{\E[\pi]}\bbE[  r_{0,k}(\tau_{-k}\cdot) x_k(\tau_{-k}\cdot)  g_\e(\cdot)\Big]
\\
	&=\frac{1}{\E[\pi]}\bbE[  r_{0,-k}(\cdot) (-x_{-k}(\cdot))  g_\e(\cdot)\Big]
	=-\bbQ_0[p_{0,-k}x_{-k}g_\e]
\end{align*}
where for the first equality we have used that $d\bbQ_0/d\bbP=\pi/\E[\pi]$ and for the second equality the translation invariance of $\bbP$. The first equality in \eqref{bello} then follows by summing over all $k\in\Z$.

By putting \eqref{bello} back into \eqref{essere}, we see that the proof of the claim is concluded if we can prove that
\begin{align}\label{vero}
\lim_{\e\downarrow0}2\bbQ_0[\varphi g_\e]=\|f+\varphi\|_{-1}^2-\|f\|_{-1}^2-\|\varphi\|_{-1}^2.
\end{align}
Note that, by spectral calculus,  the symmetric operators $( \e-\bbL_0)^{-1}$ and $(\e-\bbL_0)^{-1/2}$ are defined on the whole $L^2(\bbQ_0)$.  Since moreover  $(\e-\bbL_0) g_\e=f$, we have that
\begin{align*}
2\bbQ_0[\varphi g_\e]
	&=2\bbQ_0[\varphi (\e-\bbL_0)^{-1}f]\\
	&=
	2\langle (\e-\bbL_0)^{-1/2}\varphi, (\e-\bbL_0)^{-1/2}f \rangle\\
	&=\langle (\e-\bbL_0)^{-1/2}(\varphi+f), (\e-\bbL_0)^{-1/2}(\varphi+f) \rangle \\
	&\qquad\qquad			-\langle (\e-\bbL_0)^{-1/2}f, (\e-\bbL_0)^{-1/2}f \rangle-\langle (\e-\bbL_0)^{-1/2}\varphi, (\e-\bbL_0)^{-1/2}\varphi \rangle
\\
	&\xrightarrow{\e\downarrow 0}\|f+\varphi\|_{-1}^2-\|f\|_{-1}^2-\|\varphi\|_{-1}^2\,.
\end{align*}
The last limit follows from the observation that,
for each $V\in H_{-1}\cap L^2(\bbQ_0)$, we have
\begin{align*}
\langle (\e-\bbL_0)^{-1/2}V, (\e-\bbL_0)^{-1/2}V \rangle 
	\xrightarrow{\e\downarrow 0}\|V\|_{-1}^2.
\end{align*}
Indeed, writing 
$e_V$ for  the spectral measure associated to $V$ and $-\bbL_0$, it holds \begin{align*}
\langle (\e-\bbL_0)^{-1/2}V, (\e-\bbL_0)^{-1/2}V \rangle 
	=\int  _{[0,\infty)} \frac{1}{\e+\theta}e_V(d\theta)\xrightarrow{\e\downarrow0}\int  _{[0,\infty)} \frac {1}{\theta}e_V(d\theta)=\|V\|_{-1}^2\,.
	\end{align*}
\end{proof}


\section{Proof of Theorem \ref{teo_einstein}--(i)}\label{dim_ultimo}
We fix $\l_0 \in [0,1)$ and prove the continuity of $v_Y(\l)$, $v_{\bbY}(\l)$  at $\l_0$. To this aim we take $\l_* \in (\l_0,1)$ and  restrict below to $\l \in [0,\l_*)$. 
The positive constants $C,C'$  will depend on $\l_*$ but not on the specific choice of $\l$, moreover they   can change from line to line.

\subsection{Continuity of $v_Y(\l)$} 
We first observe  that $\lim _{\l \to \l_0} \pi^\l= \pi^{\l_0}$ $\bbP$--a.s., where $ \pi^\l(\o) := \sum _{k\in \bbZ} c_{0,k}^\l(\o)$.   Indeed, by Assumption (A4), we can bound  $|c_{0,k}^\l| \leq C {\rm  e}^{ -(1-\l_*) d |k|}
$,  $\bbP$--a.s.,    and therefore the  claim follows from dominated convergence applied to the counting measure on $\bbZ$.   

Since $ p_{0,k}^\l  =c_{0,k}^\l/\pi^\l$  and $\pi^\l \to \pi^{\l_0}$, we obtain that  
\begin{equation*}\lim _{\l \to \l_0}  p_{0,k}^\l= p_{0,k}^{\l_0} \qquad \text{ $\forall k \in \bbZ$\,,\qquad $\bbP$--a.s.}
\end{equation*}
Note that $\pi^\l \geq c_{0,1}^\l \geq C {\rm e}^{-Z_0}$. Using  also that ${\rm e}^{-(1-\l_*)u } u \leq C {\rm e}^{-\frac{(1-\l_*)}{2}u } $ for all $u\geq 0$ and using   Assumption (A4)   we get  \begin{equation}\label{CLR} p_{0,k}^\l |x_k| \leq C   {\rm e}^{ Z_0}  {\rm e}^{-|x_k|+\l x_k }|x_k| \leq C' {\rm e}^{ Z_0}    {\rm e}^{ -\frac{(1-\l_*)}{2}  |d| k }\qquad \text{$\bbP$--a.s.}
\end{equation}
We claim that  $\varphi_\l \in L^2 (\bbQ_0)$ and that  $ \lim _{\l \to \l_0} \| \varphi_\l - \varphi_{\l_0} \|_{L^2(\bbQ_0)} =0$. Indeed,  by \eqref{CLR}, we have  that $| \varphi_\l| \leq C {\rm e}^{ Z_0} $,  $\bbQ_0$--a.s. 
Since $\bbE[{\rm e }^{2 Z_0} ] <\infty$, $\pi \leq C$ $\bbP$--a.s. and $\bbQ_0[\star]= \bbE[\pi]^{-1} \bbE [\pi \star]$,  we have   that $ {\rm e}^{ Z_0} \in L^2(\bbQ_0)$. To conclude the proof of our claim it is enough to apply the dominated convergence theorem to  the measure $\bbQ_0$.

Since $\varphi_{\l_0} \in L^2(\bbQ_0)$, by Theorem \ref{serio_bis} and the Cauchy--Schwarz inequality we derive that $\varphi_{\l_0} \in L^1( \bbQ_\l)$, in particular the expectation   $\bbQ_\l[\varphi_{\l_0}]$  is well--defined.  Due to \eqref{adsl1} we can therefore  write
\begin{equation}\label{acqua}
v_Y(\l)-v_Y(\l_0)= \bbQ_\l [\varphi_\l] -\bbQ_{\l_0} [\varphi_{\l_0}] =  \bbQ_\l [\varphi_\l -\varphi_{\l_0}] +  
\bbQ_\l [\varphi_{\l_0}] -\bbQ_{\l_0}[\varphi_{\l_0}]\,.
\end{equation}
By Theorem \ref{serio_bis},   the Cauchy--Schwarz inequality and since $ \lim _{\l \to \l_0} \| \varphi_\l - \varphi_{\l_0} \|_{L^2(\bbQ_0)} =0$, we get for $\l\to \l_0$ 
\begin{equation}\label{mango1}
\bigl|  \bbQ_\l [\varphi_\l -\varphi_{\l_0}]\big|=\Big | \bbQ_0\Big [ \frac{ d \bbQ_\l}{d\bbQ_0} (\varphi_\l -\varphi_{\l_0})\Big]\Big|\leq \Big\| \frac{ d \bbQ_\l}{d\bbQ_0} \Big\|_{L^2(\bbQ_0)} \| \varphi_\l - \varphi_{\l_0} \|_{L^2(\bbQ_0)}  \to 0 \,.\end{equation}
Since we have proved that $\varphi_{\l_0} \in L^2(\bbQ_0)$, by Theorem \ref{teo_continuo} we get that $\lim_{\l \to \l_0} \bbQ_\l [\varphi_{\l_0}] =\bbQ_{\l_0}[\varphi_{\l_0}]$. By combining this last limit with \eqref{acqua} and \eqref{mango1}, we conclude that $\lim_{\l \to \l_0} v_Y(\l)=v_Y(\l_0)$.
\subsection{Continuity of $v_{\bbY}(\l)$}\label{peracotta} Due to the continuity of $v_Y(\l)$ and due to \eqref{adsl1}, it is enough to prove that the map $\l \mapsto \bbQ_\l\bigl[ 1/\pi^{\l} \bigr]$ is continuous (note that $c_{0,k}^\l= r_{0,k}^\l$, thus implying that 
$\pi^\l=\sum _{k\in\Z} r_{0,k} ^\l (\o)$).

By the observations in the above subsection we have   that  $\lim _{\l \to \l_0} \pi^\l= \pi^{\l_0}$ $\bbQ_0$--a.s. and $ 1/\pi^\l  \leq C {\rm e}^{Z_0} \in L^2(\bbQ_0)$. We get three main consequences (applying also  Theorem \ref{serio_bis} and the Cauchy--Schwarz inequality): (i) $1/\pi_\l \in L^2(\bbQ_0)$,  (ii)  $1/\pi_\l \in L^1(\bbQ_{\l_0})$ (hence the expectation $\bbQ_{\l_0}[1/\pi^\l]$ is well--defined) and  (iii) $\lim_{\l \to \l_0} \|1/ \pi_\l  - 1/\pi_{\l_0} \| _{L^2(\bbQ_0)} =0$. We then write
\begin{equation}
\bbQ_\l\bigl[ 1/\pi^{\l} \bigr] -\bbQ_{\l_0}\bigl[ 1/\pi^{\l_0} \bigr]= 
\bbQ_\l[  1/ \pi^\l  - 1/\pi^{\l_0} ] + \bbQ_\l [1/\pi^{\l_0}] -\bbQ_{\l_0}[1/\pi^{\l_0}]\,.
\end{equation}
At this point, we can proceed as done for \eqref{acqua}, replacing $\varphi_\l $ by $1/\pi^\l$.

\section{Proof of Theorem \ref{teo_einstein}--(ii)}\label{dim_einstein}
We recall that we denote by $\| \cdot \|_{-1}$ the $H_{-1}$ norm referred to the operator $-\bbL_0$ in $L^2(\bbQ_0)$ and by $\langle \cdot, \cdot \rangle$ the scalar product in $L^2(\bbQ_0)$.
\subsection{Einstein relation for $(Y_n^\l)$}
Since $  v_Y(\l) = \bbQ_\l [\varphi_\l]$ and $v_Y(0)=\bbQ_0[\varphi]=0$ we can write
\begin{equation}\label{champ1}
  \frac{ v_Y(\l) -v_Y(0)}{\l}= \frac{ v_Y(\l) }{\l}= \frac{ \bbQ_\l [\varphi_\l] }{\l}
 =  \bbQ_\l   \Big[ \frac{ \varphi_\l-\varphi }{\l}\Big]+\frac{ \bbQ_\l[ \varphi]-\bbQ_0[\varphi] }{\l}\,.
\end{equation}

\begin{Lemma}\label{H}
$\varphi  \in H_{-1}$.
\end{Lemma}
\begin{proof}
We need to show  that there exists a constant  $C>0$ such that for any $ h \in L^2 (\bbQ_0)$ it holds 
\begin{equation*}
\langle \varphi, h\rangle \leq C \langle h, - \bbL_0 h \rangle^{1/2} \,. 
\end{equation*}
The above bound is equivalent to 
\[
\bbQ_0\Big[\sum_{k \in \bbZ} x_k p_{0,k}  h \Big] \leq \frac{C}{\sqrt{2}}  \bbQ_0\Big [ \sum _{k \in \bbZ} p_{0,k}  \bigl( h (\t_k \cdot )- h  \bigr)^2\Big]^{1/2}\,,\]
 which is equivalent to (cf. $C':= C \sqrt{ \bbE[\pi] /2}$) 
 \begin{equation}\label{freddo}
   \bbE \Big[ \sum_{k \in \bbZ} x_k  c_{0,k}   h \Big] \leq C ' \bbE\Big[ \sum _{k \in \bbZ}  c_{0,k}  \bigl( h (\t_k \cdot)- h  \bigr)^2\Big]^{1/2}\,.
  \end{equation} Note that 
 \begin{equation*}
 \begin{split}
 \sum _{k\in \bbZ} \bbE \left[   x_k c_{0,k}  h \right]& =-\sum_{k\in \bbZ} \bbE \left[ 
 x_{-k}(\t_k \cdot) c_{0, -k} ( \t_k \cdot ) h \right] \\
 &= \rosso{- \sum_{k\in \bbZ} \bbE \left[  x_{-k} c_{0, -k}   h(\t_{-k} \cdot ) \right] }= - \sum_{k\in \bbZ} \bbE \left[ 
 x_k c_{0, k}  h(\t_{k} \cdot) \right] \,.
 \end{split}
 \end{equation*}
Indeed, in the first identity we have used that  $c_{0,k} (\o)= c_{0,-k} (\t_k\o)$ and $x_k(\o)= - x_{-k}(\t_k\o)$, in the second one we have used the translation invariance of $\bbP$, in the third one we have replaced $k$ by $-k$. By the above identity and  the Cauchy-Schwarz inequality we have
\begin{equation*}
\begin{split}
\text{l.h.s. of \eqref{freddo}}
	& = -\frac{1}{2} \sum_{k \in \bbZ} \bbE \big[  c_{0,k}x_k   \big(h (\t_k\cdot)-h  \big)\big]\\
	&\leq  C'' \Big( \sum_{k \in \bbZ}  \bbE \bigl[    c_{0,k}   x_k^2\bigr]\Big)^{1/2}   \Big( \sum_{k \in \bbZ} \bbE \Big[  c_{0,k}   [h (\t_k\cdot)-h ]^2 \Big] \Big)^{1/2}  \,.
\end{split}
\end{equation*}
thus concluding the proof of \eqref{freddo}.
\end{proof}

 As a consequence of Lemma \ref{H} and Theorem \ref{teo_derivo} we have   (recall definition \eqref{hf})
 \begin{equation}\label{champ2}
 \lim _{\l \to 0} \frac{ \bbQ_\l[ \varphi]-\bbQ_0[\varphi] }{\l}
 	= \partial_{\l=0} \bbQ_\l (\varphi)
	= \bbQ_0 \Bigl[ \sum_{k\in\Z} p_{0,k} (x_k -\varphi) \rosso{h^\varphi} \Bigr]\,.
 \end{equation}
Take \rosso{$\d >0$} small enough as in Lemma \ref{limitato} of Appendix \ref{poux}. 
Using \eqref{espansione} we can write, for $\l \in (0,\d)$, 
\begin{equation}\label{champ3}
\bbQ_\l   \Bigl[ \frac{ \varphi_\l-\varphi }{\l}\Bigr]= \bbQ_\l \Bigl[ \sum _{k \in \bbZ}  \partial_{\l=0} p^\l_{0,k}\,  x_k\Bigr]
+ \rosso{\frac{\l}{2}} \cE(\l)
\end{equation}
where  $\cE(\l)$ can be bounded as 
\[   \bbQ_\l \Big[ \sum _{k \in \bbZ}   \Big(\sup_{\z \in [0,\d]}  | \partial ^2_{\l=\z} p_{0,k}^\l | \Big) |x_k| \Big]
	\leq  \rosso{\sup_{\xi \in [0,\d]}\Big{\| }\frac{d\bbQ_\xi}{d\bbQ_0}}\Big{\|}_{L^2(\bbQ_0)}  \Big{  \|}  \sum _{k \in \bbZ}  \Big(\sup_{\z \in [0,\d]}  | \partial ^2_{\l=\z} p_{0,k}^\l | \Big) |x_k|\Big{\|}_{L^2(\bbQ_0)}  \,.
\]
Due to Theorem \ref{serio_bis} and \eqref{limitato2} in  Lemma \ref{limitato} in the Appendix, the above $\l$--independent upper bound is finite. Hence $\sup _{\l \in [0,\d]} |\cE(\l)| <\infty$,  thus implying that $\lim _{\l \downarrow 0} \l \cE(\l)=0$.
 On the other hand, since by Lemma \ref{limitato} in the Appendix the function $\sum _{k \in \bbZ}  \partial_{\l=0} p^\l_{0,k}\,  x_k$ belongs to $L^q(\bbQ_0)$, by Theorem \ref{teo_continuo} we get that 
\begin{equation}\label{champ5}
 \lim _{\l \downarrow 0} \bbQ_\l \Bigl[ \sum _{k \in \bbZ}  \partial_{\l=0} p^\l_{0,k}\,  x_k\Bigr]
 	= \bbQ_0 \Bigl[ \sum _{k \in \bbZ}  \partial_{\l=0} p^\l_{0,k}\,  x_k\Bigr]\,.
\end{equation}

At this point, by using that $\partial_{\l=0} p_{0,k}^\l=p_{0,k}^\l(x_k-\varphi)$ and  by combining \eqref{champ1}, \eqref{champ2}, \eqref{champ3}, the limit $\lim_{\l \downarrow 0} \l \cE(\l)=0$  and \eqref{champ5}, we conclude that $v_Y(\l)$ is derivable at $\l=0$ and that
\begin{equation}\label{champ6}
\partial_{\l=0} v_Y(\l)
	= \bbQ_0 \Bigl[ \sum_{k\in\bbZ} p_{0,k} (x_k -\varphi) (x_k+\rosso{h^\varphi}) \Bigr]\,.
\end{equation}
It remains to show that the last part of \eqref{champ6} equals $D_Y$. We manipulate \eqref{champ6} to obtain
\begin{align*}
\partial_{\l=0} v_Y(\l)
	&=\bbQ_0 \Bigl[ \sum_{k\in\Z} p_{0,k} (x_k -\varphi) \rosso{h^\varphi} \Bigr]
		+\bbQ_0\Bigl[\sum_{k\in\Z}p_{0,k}x_k^2 \Bigr]-\|\varphi\|^2_{L^2(\bbQ_0)}\\
	&=-Var(N^\varphi)
		+\bbQ_0\Bigl[\sum_{k\in\Z}p_{0,k}x_k^2 \Bigr]-\|\varphi\|^2_{L^2(\bbQ_0)}\\
	&=-2\|\varphi\|^2_{-1}+\bbQ_0\Bigl[\sum_{k\in\Z}p_{0,k}x_k^2 \Bigr]=D_Y.
\end{align*}
For the second equality we have used  the second part of Theorem \ref{teo_derivo} (i.e., equation \eqref{luna}) with the function $f=\varphi$, for the third equality we have used Lemma \ref{miomistero} with $V=\varphi$ and finally the last line follows from \cite[Thm.~2.1, Eq.~(2.28)]{demasi}.

\subsection{Einstein relation for $(\bbY_t^\l)$} The continuous time process  $\t_{\bbY_t^\l} \o$ can be obtained by  a suitable random time change from the discrete time  process $\t_{Y_n^\l} \o$ as detailed in \cite[Sec.~7]{FGS}. By using this random time change  and arguing as in  the derivation of \cite[Eq.~(4.20)]{demasi}, we get that 
$D_\bbY=\bbE[\pi] D_Y$, where $\pi$ was defined in \eqref{banana}. Since we have just proved that $D_Y=  \partial_{\l=0}v_Y(\l)$, to get the Einstein relation for $\bbY_t^\l$ it is enough to show 
that $v_\bbY(\l)$ is differentiable at $\l=0$ and moreover 
$ \partial_{\l=0}v_\bbY(\l)= \bbE[\pi] \partial_{\l=0}v_Y(\l)$. Since $v_\bbY (0)=0$, thanks to \eqref{adsl1}  and since 
$\pi^\l=\sum _{k\in\Z} c_{0,k} ^\l =  \sum _{k\in\Z} r_{0,k} ^\l $ (cf. Section \ref{peracotta}),
we can write 
\begin{equation}\label{drago}
\partial_{\l=0}v_\bbY(\l)
	=\lim_{\l\downarrow 0}\frac{v_\bbY(\l)}{\l}
	=\lim_{\l\downarrow 0}\frac{v_Y(\l)}{\l}\frac{1}{\bbQ_\l[ 1/\pi^\l  ]}\,.
\end{equation}
In Section \ref{peracotta} we have proved that the map $[0,1) \ni \l \mapsto \bbQ_\l[ 1/ \pi^\l  ]\in \bbR$ is continuous. Hence, we have $\lim _{\l \downarrow 0} \bbQ_\l[ 1/ \pi^\l  ]=  
\bbQ_0 \big[1/\pi^{\l=0}\big]=\E[\pi]^{-1} $. On the other hand
we have  just proved that $\lim_{\l\downarrow 0}\frac{v_Y(\l)}{\l}=D_Y$. Coming back to \eqref{drago} we conclude that $ \partial_{\l=0}v_\bbY(\l)= D_Y \E[\pi]^{-1} = D_{\bbY}$.


\appendix


\section{Comments on \eqref{adsl1}}\label{commentini}
Formula \eqref{adsl1} for $v_\bbY(\l)$ coincides  with \cite[Eq. (9)]{FGS}. The expression for $v_{Y}(\l)$  given in \cite[Eq. (10)]{FGS} is slightly different  from our identity $ v_Y(\l)=  \bbQ_\l \bigl[  \varphi_\l \bigr] $  in \eqref{adsl1}, since  \cite[Eq. (10)]{FGS} has been obtained from the asymptotic velocity of a third random walk (which is the \rosso{discrete--time random walk on $\bbZ$ with probability} for a jump from $i$ to $k$ given by \eqref{creperie}). Let us explain how to derive that  $ v_Y(\l)=  \bbQ_\l \bigl[  \varphi_\l \bigr] $.  We consider the process  $(\o^\l_n)$, defined as  $\o^\l_n := \t_k \o$ where $k\in \bbZ$ satisfies $x_k=Y^\l_n$. Note that, due to Assumption (A3), one recovers a.s. $(Y^\l_n)$ as  \rosso{an} additive functional of   $(\o^\l_n)$. More precisely, $Y^\l_n = \sum_{k=0}^{n-1} h ( \o_k ^\l, \o _{k+1}^\l)$, where $h(\o,\o'):= x_i $ if $ \o ' = \t_i \o$ for some $i$, and  
$h(\o,\o'):= 0 $ if $\o'$ does not coincide with any translation of $\o$.
 Let us denote by \rosso{$\bbE^\l_{\bbQ_\l}$} the expectation w.r.t. the process $(\o^\l_n)$  starting with distribution $\bbQ_\l$. Then,
 using that $\bbQ_\l$ is an ergodic distribution for the process $(\o^\l_n)$, by Birkhoff's  ergodic theorem we get that $\lim _{n\to \infty }\frac{Y_n^\l}{n}$ exists \rosso{a.s.~for  $\bbQ_\l$--a.a.~initial configurations}  and equals $\rosso{\bbE^\l_{\bbQ_\l}} \bigl[ h (\o_0, \o_1)\bigr]= \bbQ_\l [ \varphi_\l]$. 
Since, as proven in \cite{FGS}, $\bbQ_\l$ and $\bbP$ are mutually absolutely continuous, we conclude that $\lim _{n\to \infty} \frac{Y^\l_n}{n}=  \bbQ_\l [ \varphi_\l]$ \rosso{ a.s.~for  $\bbP$--a.a.~initial configurations.}

\section{Collected computations}\label{poux}
Here we collect some basic estimates that are useful in several parts of the paper. In what follows, $\l_*$ is a fixed value in $(0,1)$. All constants of the form $K, C$ appearing below (possibly with some additional typographic character)  have to be thought of as $\l_*$--dependent  but  uniform for all $ \l \in [0,\l_*]$. Moreover, the above constants can change from line to line. Moreover, without further mention,  we  will restrict to $\o$ such that $|x_k | \geq k |d|$. We recall that by Assumption (A4) this event has $\bbP$--probability one.

It is convenient to express the jump probabilities $p_{0,k}^\l(\o)$ in terms of the conductances introduced in \eqref{tropicale}. Comparing with \eqref{creperie} we can write
\begin{equation}\label{campana} p_{0,k}^\l(\o) = \frac{c_{0,k}^\l(\o)}{\pi^\l (\o)}\,,\qquad \pi^\l(\o) := \sum _{j\in \bbZ} c_{0,j}^\l(\o)\,.\end{equation}
Note that $\pi ^\l=\pi $ when $\l=0$ (cf.~ \eqref{banana}).

An easy calculation shows that
\begin{align}
\partial_{\l} p_{0,k}^\l&=p_{0,k}^\l(x_k-\varphi_\l)\label{derivata1}\\
\partial^2_{\l} p_{0,k}^\l&=p_{0,k}^\l\Big(x_k^2-2x_k\varphi_\l+2\varphi_\l^2-\sum_{j\in\Z}p_{0,j}^\l x_j^2\Big)\,.\label{derivata2}
\end{align}
We also observe that, for some universal constant $c$, it holds  
\begin{equation}\label{derivata3}
\big| \partial^2_{\l} p_{0,k}^\l \big| 
	\leq  c \, p_{0,k}^\l\Big( x_k ^2  +\sum_{j\in\Z}p_{0,j}^\l x_j^2\Big)\,.
\end{equation}
Indeed,  \rosso{by \eqref{derivata2}} we can bound 
\[\big| \partial^2_{\l} p_{0,k}^\l \big| \leq c'  p_{0,k}^\l\Big(x_k^2  +\varphi_\l^2+\sum_{j\in\Z}p_{0,j}^\l x_j^2\Big)\]
 for some universal constant $c'$. On the other hand, by the Cauchy--Schwarz inequality, $ \varphi_\l^2 \leq  \sum_{j\in\Z}p_{0,j}^\l x_j^2$. 
We also have that, for some finite constant $C>0$, 
\begin{align}\label{lega}
\frac{ p^\l_{0,k}  }{p_{0,k}}={\rm e}^{\l x_k}\frac{\pi}{\pi^\l} \leq C\,{\rm e}^{\l (x_k+ Z_{-1})}  \qquad \forall k \in \bbZ\,, \qquad \forall \l \in [0,\l_*]\,.
\end{align}
This is true since $c^\l_{-1,0}+c^\l_{0,1}\leq \pi^{\l}\leq K(c^\l_{-1,0}+c^\l_{0,1})$  for some constant $K$   (see \cite[Rem.~3.2]{FGS}, \cite[Lemma 3.6]{FGS} and  Remark \ref{ricorda}), and therefore
\begin{align}\label{wehavepie}
\frac{\pi}{\pi^\l}
	\leq  K'
	 \frac{  {\rm e}^{-Z_{-1}}  +  {\rm e}^{-Z_{0}}}{{\rm e}^{-(1+\l)Z_{-1}}+{\rm e}^{-(1-\l)Z_{0}}}
	\leq K' \Big(1+\frac{{\rm e}^{-Z_{-1}}}{{\rm e}^{-(1+\l)Z_{-1}}}\Big)
	\leq C {\rm e}^{\l Z_{-1}}.
\end{align}

Another bound which will be repeatedly   used  below is the following. For a  fixed  positive integer  $n$, it holds
\begin{equation}\label{sole100}
\sum _{k \in \bbZ} p_{0,k} ^{\l } \rosso{|x_k|^ n} \leq C \frac{1}{\pi^{\l }}   \sum_{k\in\Z} {\rm e}^{-|x_k|+\l  x_k}    \rosso{|x_k|^n}\leq \tilde C\frac{1}{\pi^{\l }}\,,  \qquad \forall \l \in [0,\l_*]
\end{equation}
(\rosso{$\tilde{C}$} depends on $\l_*$ and $n$). Above we used that
${\rm e}^{-(1-\l_*) u }  u^n \leq C  {\rm e}^{-(1-\l_*) u/2 }$ for all $u \geq 0$ and  \rosso{that} $|x_j| \geq d j$\rosso{.} As a consequence of \eqref{sole100} we get
 \begin{equation}\label{luna100}
\rosso{| \varphi_\l|^n }\leq \sum _{k\in \bbZ} p_{0,k}^\l \rosso{|x_k|^n}  \leq\frac{C}{\pi^{\l }} \,, \qquad \forall \l \in [0,\l_*]\,.
 \end{equation} 
Since $d\bbQ_0/d\bbP=\pi/\bbE[\pi]$, by \eqref{wehavepie}, \eqref{sole100} and \eqref{luna100} we get
\begin{equation}\label{passetto}
 \bbE[ {\rm e}^{Z_0}]<\infty \; \Longrightarrow \sup _{\l \in [0,\l_*]} \bbQ_0\Big[\sum_{k\in\Z} p_{0,k}^\l |x_k|^n \Big]<\infty \text{ and }  \sup _{\l \in [0,\l_*]} \bbQ_0\Big[|\varphi_\l|^n \Big]<\infty\,.
 \end{equation}
%


%
%
%
%

\begin{Lemma}\label{limitato} 
Suppose $\bbE[{\rm e}^{p Z_0}]< \infty$ for some $p> 2$,  let $q>1 $ be such that $p^{-1}+q^{-1}=1$ and let $\hat q >2$ be such that   \rosso{$p^{-1}+ {\hat  q}^{-1} =2^{-1}$}. Then, for $\d$ small enough, it holds
\begin{align}
& \sum _{k \in \bbZ} | \partial_{\l=0} p^\l_{0,k}\cdot  x_k| \in L^q(\bbQ_0)\subset L^1(\bbQ_\l)\,,\label{limitato1}\\
& \sum _{k \in \bbZ}   \Big(\sup_{\z \in [0,\d]}  | \partial ^2_{\l=\z} p_{0,k}^\l | \Big) |x_k|\in L^2(\bbQ_0)\subset L^1(\bbQ_\l)\,,\label{limitato2}\\
& \sum _{k \in \bbZ\setminus\{0\}  }  \rosso{(p_{0,k})} ^{1- \hat{q} }  \Big(\sup_{\z \in [0,\d]}  | \partial ^2_{\l=\z} p_{0,k}^\l | \Big) ^{\hat{q} } 
 \in L^1(\bbQ_0)\,.\label{limitato3}
\end{align}
\end{Lemma}
\begin{proof}
\rosso{Since $p>2$ we have $q \in (1,2)$, thus implying that $L^2( \bbQ_0 ) \subset L^q(\bbQ_0)$ by the H\"older inequality.  To get the set  inclusions stated in the lemma, it is therefore enough to check that $L^q(\bbQ_0) \subset L^1(\bbQ_\l)$. This  can be easily checked by writing $\bbQ_\l[\star]=\bbQ_0[\star\cdot d\bbQ_\l/d\bbQ_0]$, using the H\"older  inequality and then Theorem \ref{serio_bis}. }

\smallskip
We call $f_1,f_2$ and $ f_3$ the l.h.s. of \eqref{limitato1}, \eqref{limitato2} and \eqref{limitato3}, respectively.
For \eqref{limitato1} we use  \eqref{derivata1} 
and the Cauchy-Schwarz inequality to bound
\begin{align*}
\rosso{\| f_1 \|^q_{L^q(\bbQ_0)}}
	\leq \bbQ_0\Big[ \Big(\sum_{k\in\bbZ}p_{0,k}(x_k-\varphi)^2\Big)^{q/2}\Big(\sum_{k\in\Z}p_{0,k}x_k^2\Big)^{q/2} \Big].
\end{align*}
As in the proof of Lemma \ref{portamivia} we take $A:=2/q> 1$ (recall that $p >2$)  and $B:=2/(2-q)$ (so that $A^{-1}+B^{-1}=1$) and use the H\"older inequality to further obtain
\begin{align*}
\rosso{\| f_1 \|_{L^q(\bbQ_0)}^q}
	\leq \bbQ_0\Big[ \Big(\sum _{k\in \bbZ}    p_{0,k} (x_k-\varphi)^2\Big)^{\frac{qB}{2}}\Big]^{1/B}
			 \bbQ_0\Big[ \sum _{k\in \bbZ}    p_{0,k}x_k^2\Big]^{1/A}.
\end{align*}
The first term in the r.h.s.~can be bounded as in \eqref{rosso}, the second is bounded by \eqref{passetto}. 
\smallskip

We move to \eqref{limitato2}. To prove that $f_2 \in L^2( \bbQ_0) $ we need to show that $\bbE[ \pi f_2^2 ] < \infty$. We take $\d  $ small (the precise value will be stated at the end) and set $\l_*:=\d$ (hence,  our $C$--type constants  below depend on $\d$ but not on the specific $\l \in [0, \d]$).
We note that for all $\z \in [0,\d]$ it holds 
\begin{equation}\label{vapore}
\begin{split}
 | \partial ^2_{\l=\z} p_{0,k}^\l |  |x_k|  
 	& \leq C  p_{0,k}^\z \Big(|x_k| ^3  +  \big( \sum _{j\in \bbZ} p_{0,j} ^\z x_j^2\big) ^2 \Big) 
	\leq C  p_{0,k}^\z \Big ( |x_k| ^3  +   \sum _{j\in \bbZ} p_{0,j} ^\z x_j^4  \Big) \\ 
	& \leq C' p_{0,k} {\rm e}^{ \d ( |x_k|+ Z_{-1} ) } \Big ( |x_k| ^3  +   \sum _{j\in \bbZ} p_{0,j}  {\rm e}^{ \d ( |x_j|+ Z_{-1} ) }  x_j^4  \Big) \,.
%
%
 \end{split}
  \end{equation}
  Indeed,  the first inequality follows  from  \rosso{ \eqref{derivata3} and the property }that $|x_k| \geq d$ for $k \not =0$ \rosso{(as intermediate step bound  the product $(\sum_j p_{0,j}^\z  x_j^2) |x_k|$ by  the sum of their squares)}. The second inequality follows  from the Cauchy--Schwarz inequality, while the third inequality follows from  
 \eqref{lega}.

 Note that the last term of \eqref{vapore} depends only on $\d$. Hence, 
  to  prove that 
   $\bbE[ \pi f_2^2 ] < \infty$, we only need to show that (we use repeatedly  the Cauchy--Schwarz inequality)
   \begin{align}
   & \bbE\Big[ \pi \sum _{k \in \bbZ}  p_{0,k} {\rm e}^{ 2\d ( |x_k|+ Z_{-1} ) }  |x_k| ^6 \Big]<\infty \label{zeus1} 
  \\
  &  \bbE\Big[ \pi \sum _{k \in \bbZ}  p_{0,k} {\rm e}^{ 2\d ( |x_k|+ Z_{-1} ) }  
     \sum _{j\in \bbZ} p_{0,j}  {\rm e}^{2 \d ( |x_j|+ Z_{-1} ) }  x_j^8  \Big] <\infty\,.\label{zeus2}
   \end{align}
   We prove \eqref{zeus2}, the proof of \eqref{zeus1} follows the same lines and it is even simpler.
   Using that ${\rm e} ^{-(1-2\d)  u} (1+u^8) \leq C  \rosso{{\rm e}^{- u/2}}$  for all $u \geq 0$ if we restrict to  \rosso{$ \d \leq 1/8$},  we can bound the integrand in \eqref{zeus2} by 
   \[  \frac{C}{\pi} \sum _{k \in \bbZ}  {\rm e}^{- \frac{|x_k|}{2} }  {\rm e}^{ 2\d  Z_{-1}  }  
     \sum _{j\in \bbZ}{\rm e}^{- \frac{|x_j|}{2} }   {\rm e}^{2 \d  Z_{-1}  }\,. \]
   Since $|x_k | \geq d |k|$ and since $\pi \geq c_{-1,0} \geq C {\rm e}^{- (1+ \d) Z_{-1} }$, we conclude that the $\bbP$--expectation of \eqref{zeus2} is finite if $\bbE [ {\rm e}^{(1+ 5 \d) Z_{-1}}]< \infty$. By taking $\d$ small enough, the last  bound is satisfied \rosso{due to the assumption $\bbE[{\rm e}^{pZ_0}]<\infty$}.
   
   \smallskip

We move to \eqref{limitato3}.  Again we need to prove that $\bbE[ \pi f_3] < \infty$. Similarly to \eqref{vapore},  \rosso{by   \eqref{derivata3}  and \eqref{lega}},   we get
\begin{equation*}
\begin{split}
 | \partial ^2_{\l=\z} p_{0,k}^\l |    
	\leq C  p_{0,k}^\z \Big ( |x_k| ^2  +   \sum _{j\in \bbZ} p_{0,j} ^\z x_j^2  \Big) 
    \leq  C' p_{0,k} {\rm e}^{ \d ( |x_k|+ Z_{-1} ) } \Big ( |x_k| ^2 +   \sum _{j\in \bbZ} p_{0,j}  {\rm e}^{ \d ( |x_j|+ Z_{-1} ) }  x_j^2  \Big) \,.
 \end{split}
  \end{equation*}
  Then, using also that   $(x+y) ^{\hat q} \leq c (\hat q) ( x^{\hat q} + y^{\hat q})$ for all $x, y \geq 0$ and the H\"older inequality,
  \begin{equation}
  \begin{split}
    f_3 \leq  C  \sum _{\rosso{ k \in \bbZ}
     } p_{0,k}   {\rm e}^{ \hat{q} \d ( |x_k|+ Z_{-1} ) }  |x_k| ^{2\hat q}  +  \sum _{ \rosso{k \in \bbZ} } p_{0,k}   {\rm e}^{ \hat{q} \d ( |x_k|+ Z_{-1} ) }     \sum _{j\in \bbZ} p_{0,j}   {\rm e}^{\hat q  \d ( |x_j|+ Z_{-1} ) }  x_j^{2 \hat q}\,.
  \end{split}
  \end{equation}
At this point, we get that $\bbE[ \pi f_3] < \infty$ if we prove 
\begin{align}
&\bbE \Big[ \pi  \sum _{ \rosso{k} \in \bbZ } p_{0,k}   {\rm e}^{ \hat{q} \d ( |x_k|+ Z_{-1} ) }  |x_k| ^{2\hat q}  \Big] <\infty\,, \label{zeus3}\\
&\bbE \Big[  \pi  \sum _{ \rosso{k}\in \bbZ  } p_{0,k}   {\rm e}^{ \hat{q} \d ( |x_k|+ Z_{-1} ) }     \sum _{j\in \bbZ} p_{0,j}   {\rm e}^{\hat q  \d ( |x_j|+ Z_{-1} ) }  x_j^{2 \hat q} \Big] < \infty \,.\label{zeus4}
\end{align}
The above bound can be proved by the same arguments adopted for  \eqref{zeus2}
when $\d$ is small enough.
 \end{proof}

\begin{Lemma}\label{tuttabirra} Suppose $\bbE[{\rm e}^{p Z_0}]<\infty$ for some $p>1$.  Given $\l_0 \in [0,1)$, it holds 
\begin{equation}
 \lim _{ \l \to \l_0} \bbQ_0 \Bigl[ \big( \sum _{k\in\bbZ} | p_{0,k}^\l - p_{0,k}^{\l_0}  |\bigr)^4 \Big] 
 	=0\,.
\end{equation}
\end{Lemma}
\begin{proof} We fix $\l_*\in (\l_0,1)$. \rosso{Recall that} all constants of type $C,K$ appearing in what follows can depend on $\l_*$ but do not depend on the particular bias parameter taken in $[0,\l_*]$, and moreover can change from line to line.
First of all we bound, by applying the H\"older  inequality,
\begin{equation}\label{grotto}
 \bbQ_0 \Bigl[ \big( \sum _{k\in\bbZ} | p_{0,k}^\l - p_{0,k}^{\l_0} |\bigr)^4 \Big] 
 	=\bbQ_0 \Bigl[ \Big( \sum _{k\in\bbZ\setminus\{0\}} p_{0,k}^{\l_0}\Big| \frac{p_{0,k}^\l - p^{\l_0}_{0,k}}{p^{\l_0}_{0,k}} \Big|\Bigr)^4 \Big]
	\leq  \bbQ_0 \Bigl[  \sum _{k\in\bbZ\setminus\{0\}} p_{0,k}^{\l_0}\Big| \frac{p_{0,k}^\l - p_{0,k}^{\l_0}}{p_{0,k}^{\l_0}} \Big|^4 \Big]\,.
\end{equation}
By the Taylor expansion with the Lagrange rest at the first order and by \eqref{derivata1} we have
\begin{align*}
p_{0,k}^\l - p^{\l_0}_{0,k}=(\l-\l_0) \,\partial_{\l=\xi_k}p_{0,k}^\l=(\l-\l_0)\,p_{0,k}^{\xi_k}(x_k-\varphi_{\xi_k}),
\end{align*}
where $\xi_k$ is some random  value   between $\l_0$ and $\l$ depending on  $k$, $\l_0$  and $\l$. Therefore we can continue from \eqref{grotto} and bound  
\begin{equation}\label{pontiac}
\bbQ_0 \Bigl[  \sum _{k\in\bbZ\setminus\{0\}} \rosso{p_{0,k}^{\l_0} }\Big | \frac{p_{0,k}^\l - p_{0,k}^{\l_0} }{p_{0,k}^{\l_0}} \Big|^4 \Big]
	\leq C(\l-\l_0)^4\Big(\bbQ_0 \Bigl[  \sum _{k\in\bbZ\setminus\{0\}}\frac{(p_{0,k}^{\xi_k})^4}{( p_{0,k}^{\l_0})^3}x_k^4 \Big]
		+\bbQ_0 \Bigl[  \sum _{k\in\bbZ\setminus\{0\}}\frac{(p_{0,k}^{\xi_k})^4}{( p_{0,k}^{\l_0})^3}\varphi_{\xi_k}^4 \Big]\Big).
\end{equation}
Given $\d>0$ small (the precise value of $\d$ will be stated below) we set $U_\d:= [\l_0-\d, \l_0+\d] $ and assume $ U_\d \subset  [0, \l_*]$. 
If we show that both the $\Q_0$-expectations on the r.h.s. of \eqref{pontiac} are finite uniformly in $\l \in U_\d$, then we are done.
To this aim  we extend the bound in \eqref{lega}. Indeed, by the same arguments used for \eqref{lega}, we have 
for any $\l,\z\in [0,\l_*]$ and $k \in \bbZ$ that 
\begin{equation}\label{lega100}
\frac{ p^\l_{0,k}  }{p^\z_{0,k}}={\rm e}^{(\l-\z) x_k}\frac{\pi^\z }{\pi^\l} \leq
C  {\rm e}^{(\l-\z) x_k}\frac{   {\rm e}^{-(1-\z)Z_{0}}+{\rm e}^{-(1+\z)Z_{-1}}}{{\rm e}^{-(1-\l)Z_{0}}+{\rm e}^{-(1+\l)Z_{-1}}}
\leq C  {\rm e}^{|\l-\z|\cdot | x_k|}  \Big[ {\rm e } ^{|\l-\z| Z_0}+ {\rm e } ^{|\l-\z| Z_{-1}}\Big]
\end{equation}
(the above constant $C$ does not depend on $k\in \bbZ$).

From now on we restrict to $\l \in U_\d$ (thus implying that $\xi_k \in U_\d$). 
Then by \eqref{lega100} we can bound 
\[
\frac{(p_{0,k}^{\xi_k})^4}{( p_{0,k}^{\l_0})^4}x_k^4 \leq 
\rosso{C {\rm e} ^{ 4 \d |x_k|} \big[ {\rm e } ^{4\d  Z_0}+ {\rm e } ^{4\d Z_{-1}}\big] x_k^4}\leq 
 C' {\rm e} ^{ 5 \d |x_k|} \big[ {\rm e } ^{4\d  Z_0}+ {\rm e } ^{4\d Z_{-1}}\big] \]
 (\rosso{$C'$} depends on $\d$). Hence we get (cf.~\eqref{wehavepie})
\begin{equation}\label{vento}
\frac{ d\bbQ_0 }{d \bbP} \frac{(p_{0,k}^{\xi_k})^4}{( p_{0,k}^{\l_0})^3}x_k^4 =
\bbE [\pi]^{-1}\frac{\pi}{\pi^{\l_0} } c_{0,k}^{\l_0} \frac{(p_{0,k}^{\xi_k})^4}{( p_{0,k}^{\l_0})^4}x_k^4  
\leq C' {\rm e}^{\l_0 Z_{-1} } {\rm e}^{- (1-\l_0-5\d ) |x_k|}   {\rm e } ^{4\d  Z_0+4 \d Z_{-1} }\,.
\end{equation}
We assume $\d$  \rosso{ so small}  that $\l_0+5\d<1$.
Using that $ |x_k| \geq k d$, to prove that the  first expectation in the r.h.s.~of \eqref{pontiac} is bounded uniformly in $\l\in U_\d$ we only need to show that 
\begin{equation}\label{martello}
\rosso{\bbE[ {\rm e}^{(\l_0+4\d) Z_{-1}+3 \d  Z_0}]<\infty\,.}
\end{equation} Before explaining how to proceed   
 we move to   the second $\Q_0$-expectation on the last line of \eqref{pontiac}. Due to \eqref{luna100} and since $\pi ^{\xi_k} \geq c_{0,1}^{\xi_k} \geq C {\rm e}^{ -(1-\xi_k) Z_0} $, we have 
\begin{align*}
 \varphi_{\xi_k}^4 \leq \frac{C}{\pi^{\xi_k}}\leq \tilde C {\rm e}^{(1-\l_0-\d) Z_{0}}\,.
\end{align*}
Reasoning as in \eqref{vento} we get 
\[ \frac{ d\bbQ_0 }{d \bbP} \frac{(p_{0,k}^{\xi_k})^4}{( p_{0,k}^{\l_0})^3}   \varphi_{\xi_k}^4  \leq  C' {\rm e}^{\l_0 Z_{-1} } {\rm e}^{- (1-\l_0-4\d ) |x_k|}   {\rm e } ^{4\d  Z_0+ 4\d Z_{-1} } {\rm e}^{(1-\l_0-\d) Z_{0}}
\]
and the  second $\Q_0$-expectation on the last line of \eqref{pontiac}  is bounded uniformly in $\l \in U_\d$ if we prove that 
\begin{equation}
\label{chiodo}
\rosso{\bbE\big[ {\rm e}^{(\l_0+4\d) Z_{-1}+(1-\l_0+3\d)  Z_0}\big]<\infty}\,.
\end{equation}
We explain how  to get \eqref{chiodo} (\rosso{indeed,  \eqref{chiodo} implies \eqref{martello}}). By the H\"older inequality, given $a,b \geq 1$ with $ a^{-1} + b^{-1} =1$,   \eqref{chiodo} is satisfied if the expectations \rosso{$\bbE[ {\rm e}^{a(\l_0+4 \d) Z_{-1}}  ]$} and \rosso{$ \bbE[ {\rm e}^{b (1-\l_0+3 \d)   Z_0}]$} are finite. To conclude we take \rosso{$a:= (\l_0+4 \d)^{-1}$} and therefore \rosso{$b := (1-\l_0-4\d)^{-1}$}, and take $\d$ small to have \rosso{$b (1-\l_0+3\d)  \leq p$}. At the end, it remains to invoke the bound
$\bbE[ {\rm e} ^{p  Z_0}] < \infty$.
\end{proof}

\section{Proof of Lemma \ref{pizza}} \label{appendino}
To simplify the notation, inside the proof we write  $\|\cdot\|$ and $ \langle \cdot, \cdot \rangle$ for the norm and the scalar product in $L^2(Q_0)$. 
Note that $\rho_0 \equiv 1$. Since
$Q_\l(f)= Q_0( \rho_\l f)= \langle \rho_\l, f\rangle$, the $L^2$--weak convergence $\rho _\l 
\rightharpoonup \rho_0$ would imply \eqref{lego_city}.  Hence, we only  need to prove that  $\rho _\l 
\rightharpoonup \rho_0$.

Suppose by contradiction that $\rho _\l 
\not \rightharpoonup \rho_0$.
Then we can extract a sequence \rosso{$\l_n \to \l_0$} such that  $\rho_{\l_n}\not\in U$, with $U$ being a suitable open neighbourhood of 
$\rho_0$.
Let $ R:= \sup _{ \l \in I} \| \rho _\l \|$ and set $B(0,R):= \{ f \in L^2(Q_0)\,:\, \| f\| \leq R\}$. Note that $R <\infty$ by
(H1).
 By Kakutani's theorem the ball $B(0,R)$ is compact in the $L^2$--weak topology, hence the set 
 $\{ \rho _{\l_n}  \}$ is relatively compact in the $L^2$--weak topology.   As a consequence, at \rosso{the cost of extracting} a subsequence, we have that $\rho_{\l_n} \rightharpoonup  \rho$ for some $\rho\in L^2(Q_0)$. Since $\rho_{\l_n}\not \in U$, we also have that $\rho\not =\rho_0$. To get a contradiction, we prove that it must be $\rho=\rho_0$.

To this aim    we  first isolate some properties of  $\rho$.
For any function $f \in L^2(Q_0) $ with $f \geq 0 $, we have $\langle \rho, f \rangle \geq 0$ (indeed $\langle \rho_n, f \rangle \geq 0 $ since $\rho_n \geq 0$). As a consequence $ \rho \geq 0$. Moreover $\langle \rho, 1 \rangle = \lim _{n \to \infty} \langle \rho_n, 1 \rangle= 1$. By the above properties $dQ:= \rho dQ_0$ is a well--defined probability measure and $ \frac{d Q}{\, d Q_0} \in L^2(Q_0)$.  We claim  that $Q( L_0 f) =0$ for any $f \in \cC$. 
By (H2), assuming our claim,  we obtain  that $Q=Q_0$, thus implying that $\rho=\rho_0$ and leading to the contradiction. 

It remains to prove the claim.
Note that  for $f \in \cC$
\begin{equation}\label{gnocchi1}
Q( L_0 f) = \langle \rho, L_0 f \rangle = \lim _{n \to \infty} \langle \rho_{\l_n}, L_0 f \rangle = \lim_{n \to \infty } Q_{\l_n} (L_0 f) \,.
\end{equation}
Since $Q_{\l_n } (L_{\l_n} f) =0$ by (H3), using assumptions (H1) and  (H4) we can bound
\begin{equation} \label{gnocchi2} \bigl| Q_{\l_n} (L_0 f) \bigr |=  \bigl| Q_{\l_n} (L_0f- L_{\l_n}  f) \bigr |=
\bigl| Q_0 ( \rho_{\l_n} (L_0f- L_{\l_n}  f) ) \bigr| \leq \| \rho _{\l_n}\| \, \| L_0 f -L_{\l_n} f \|  \to 0 
\end{equation}
as $n\to \infty$.
As a byproduct of \eqref{gnocchi1} and \eqref{gnocchi2} we get that $Q(L_0 f)=0$ for any $f \in \cC$, thus proving our claim.


\bigskip

\noindent
{\bf Acknowledgements}. We thank S. Olla for very inspiring discussions. His talk  on the Einstein relation   at the workshop ``Random Motion in Random Media" at Eurandom (2015) has  been of remarkable help in the development of Sections \ref{dim_teo_derivo} and \ref{dim_einstein}. 
We also thank
 M. Loulakis for pointing out some relevant references.
Part of this work was done during  A.F.'s stay at the Institut Henri Poincar\'e (Centre Emile Borel) during the trimester  ``Stochastic Dynamics Out of Equilibrium''. A.F. thanks   the Institut Henri Poincar\'e   for the kind hospitality and  the support. 


\end{document}